\documentclass[11pt,a4paper, reqno]{amsart}
\usepackage[dvipsnames]{xcolor}
\usepackage{amsmath,amssymb,amsfonts,amsthm,mathtools,mathrsfs,bm}
\usepackage[shortlabels]{enumitem}
\usepackage{tikz}
\usepackage{cite}
\usepackage[pagebackref=true,colorlinks=true,breaklinks=true]{hyperref}
\hypersetup{
  linkcolor=MidnightBlue,
  citecolor=MidnightBlue,
  urlcolor=MidnightBlue,
  linktoc=all,
  pdftitle={Sharp convex body sparse bounds and quantitative matrix weighted estimates for Hilbert transforms along curves},
  pdfauthor={Xi Cen}
}
\usepackage{microtype}
\usepackage{needspace}
\renewcommand*{\backref}[1]{}

\renewcommand*{\backrefalt}[4]{%
  \ifcase #1\relax
  \or\space{\footnotesize #2}
  \else\space{\footnotesize #2}
  \fi}

\allowdisplaybreaks[3]

\newtheorem{alphatheorem}{Theorem}

\newtheorem{theorem}{Theorem}[section]
\newtheorem{proposition}[theorem]{Proposition}
\newtheorem{lemma}[theorem]{Lemma}
\newtheorem{corollary}[theorem]{Corollary}

\newtheorem{remark}[theorem]{Remark}

\newtheorem{question}{Question}
\newtheorem{introdefinition}{Definition}[section]

\numberwithin{equation}{section}

\def\R{{\mathbb R}}
\def\C{{\mathbb C}}
\def\Z{{\mathbb Z}}
\def\1{{\mathbf 1}}
\newcommand{\Sc}{\mathcal S}
\newcommand{\Qc}{\mathcal Q}
\newcommand{\Kc}{\mathcal{D}_{\gamma}}

\newcommand{\Oc}{\Omega}
\newcommand{\CB}{\mathrm{cb}}
\newcommand{\pv}{\operatorname{p.v.}}
\newcommand{\supp}{\operatorname{supp}}

\newcommand{\op}{\operatorname{op}}
\newcommand{\inte}{\operatorname{int}}

\providecommand{\eqsim}{\mathrel{\simeq}}
\newcommand{\llangle}{\left\langle\!\left\langle}
\newcommand{\rrangle}{\right\rangle\!\right\rangle}

\newcommand{\Hgamloc}{H_{\gamma,\mathrm{loc}}}
\newcommand{\Hvec}{\overrightarrow{H_{\gamma}}}
\newcommand{\Hvecloc}{\overrightarrow{H_{\gamma,\mathrm{loc}}}}
\newcommand{\vecT}{\overrightarrow T}
\newcommand{\vecf}{\mathbf f}
\newcommand{\vecg}{\mathbf g}
\newcommand{\vecF}{\mathbf F}
\newcommand{\vecG}{\mathbf G}
\newcommand{\vecGz}{\mathbf G_0}
\newcommand{\vech}{\mathbf h}

\DeclareRobustCommand{\Thmref}[1]{\mbox{\hyperref[#1]{Theorem~\ref*{#1}}}}
\DeclareRobustCommand{\Propref}[1]{\mbox{\hyperref[#1]{Proposition~\ref*{#1}}}}
\DeclareRobustCommand{\Lemref}[1]{\mbox{\hyperref[#1]{Lemma~\ref*{#1}}}}
\DeclareRobustCommand{\Corref}[1]{\mbox{\hyperref[#1]{Corollary~\ref*{#1}}}}
\DeclareRobustCommand{\Remref}[1]{\mbox{\hyperref[#1]{Remark~\ref*{#1}}}}
\DeclareRobustCommand{\Defref}[1]{\mbox{\hyperref[#1]{Definition~\ref*{#1}}}}
\DeclareRobustCommand{\Secref}[1]{\mbox{\hyperref[#1]{Section~\ref*{#1}}}}
\DeclareRobustCommand{\Figref}[1]{\mbox{\hyperref[#1]{Figure~\ref*{#1}}}}
\DeclareRobustCommand{\Extref}[2]{\mbox{\hyperlink{cite.#1}{#2}}}
\begin{document}

\title[Sharp convex body sparse bounds...]{\bf Sharp convex body sparse bounds and quantitative matrix weighted estimates for Hilbert transforms along curves}

\author[X. Cen]{Xi Cen}
\address{Xi Cen, School of Science, China University of Mining and Technology-Beijing, Beijing 100083, People's Republic of China}
\email{xicenmath@gmail.com}

\begin{abstract}
We prove convex body sparse bounds for coordinatewise Hilbert transforms along monomial curves and local $C^d$ curves with uniformly nonvanishing torsion on the closed integration interval. Adapting the scalar estimates of Cladek--Ou \cite{CladekOu2018} and the lifting method of Laukkarinen--Lorist \cite{LaukkarinenLorist2026} to the anisotropic geometry, we obtain one sparse family of genuine curve-adapted boxes and quantitative $L^p(W)$ bounds for matrix weights. The convex body sparse exponent range is sharp up to its boundary: estimates hold in the interior of $\Omega_d$, while any fixed finite number of sparse families fails outside $\overline{\Omega_d}$ when $1/r+1/s>1$. We also establish an explicit endpoint-contact obstruction. For every interior planar exponent pair, a polynomial curve with positive torsion on the open integration interval violates the full-interval sparse estimate.
\end{abstract}

\maketitle

\begin{center}
\begin{minipage}{120mm}
\end{minipage}
\end{center}

\medskip

\section{Introduction}\label{sec:introduction}

\subsection{Background and motivations}\label{subsec:background}
\vspace{0.1em}\noindent
We establish convex body sparse bounds and quantitative matrix weighted estimates for Hilbert transforms along curves. The scalar model is
\begin{align}\label{eq:Hgamma-global}
        H_\gamma f(x)\coloneqq \pv\int_\R f(x-\gamma(t))\,\frac{dt}{t},
        \qquad x\in\R^d.
\end{align}
For local curves we restrict the integral to a compact interval. For $\vecf=(f_1,\ldots,f_n):\R^d\to\C^n$, the coordinatewise extension is $\vecT\vecf=(Tf_1,\ldots,Tf_n)$. In particular,
\begin{align}\label{eq:vector-extension-H}
\Hvec\vecf=(H_\gamma f_1,\ldots,H_\gamma f_n).
\end{align}
Here $d$ is the spatial dimension and $n$ is the vector dimension.

Curve Hilbert transforms exhibit cancellation, anisotropic scaling and $L^p$ improvement \cite{SteinWainger1978,ChristStein1987,Christ1988,Christ1998Case,TaoWright2003}. Sparse domination and related weighted estimates appear in \cite{Hytonen2012,Lerner2016,BernicotFreyPetermichl2016,CondeAlonsoCuliucDiPlinioOu2017,CuliucDiPlinioOu2018,LaceySpherical2017,Oberlin2019,LiWangZhai2023}. Cladek--Ou \cite{CladekOu2018} determined the scalar sparse exponent region for monomial and local torsion curves. We distinguish torsion nonvanishing in the open interval from uniform nondegeneracy at its endpoints.

Scalar sparse averages lose vector directions needed for matrix weights. Convex body averages retain them \cite{NazarovPetermichlTreilVolberg2017,PottStoica2017,CuliucDiPlinioOu2018b,Laukkarinen2025,Nieraeth2025,KakaroumpasNguyenVardakis2026}. We apply this framework to curve-adapted sparse forms.

The abstract tools are the multi-scale theory of \cite{Laukkarinen2025}, the scalar-to-convex-body lifting of \cite{LaukkarinenLorist2026}, and the matrix sparse embedding of \cite{KakaroumpasNguyenVardakis2026}. The new geometric steps are the anisotropic reduction, the nesting $Q^+\subseteq P^+$ for $Q\subseteq P$ in \eqref{eq:adapted-plus-nesting}, the Carleson transfer \eqref{eq:adapted-carleson-transfer}, and the full ball--box comparison \eqref{eq:matrix-ball-box-equivalence-r31}. These permit a single family of genuine curve-adapted boxes and a matrix characteristic on all translated adapted boxes. The abstract lifting and matrix embedding are used as cited, rather than reproved.

Monomial $\gamma$-cubes have side-lengths $\ell^{\alpha_j}$ and unbounded eccentricity as $\ell\to0$. Our construction respects this geometry. A separate endpoint-contact test gives an obstruction despite torsion nonvanishing at every interior parameter.

Related matrix weighted settings are studied in \cite{KakaroumpasNieraeth2026,BenitoEtAl2026,ClaroEtAl2026}.

We now fix the geometry used in the main statements. Let $d\ge2$. A monomial curve is a curve $\gamma:\R\to\R^d$ of the form
\begin{align}\label{eq:monomial-curve}
\gamma(t)=
\begin{cases}
(\varepsilon_1 |t|^{\alpha_1},\ldots,\varepsilon_d |t|^{\alpha_d}),& t\ge0,\\
(\varepsilon'_1 |t|^{\alpha_1},\ldots,\varepsilon'_d |t|^{\alpha_d}),& t<0,
\end{cases}
\end{align}
where $0<\alpha_1<\cdots<\alpha_d<\infty$, $\varepsilon_j,\varepsilon'_j\in\{1,-1\}$, and $\varepsilon_j\ne\varepsilon'_j$ for at least one $j$. In proofs involving dyadic generations we normalize $\alpha_1=1$. Indeed, the change $u=\operatorname{sgn}(t)|t|^{\alpha_1}$ gives $dt/t=du/(\alpha_1u)$ and transforms the powers into $\alpha_j/\alpha_1$. The transformed Hilbert integral carries the factor $\alpha_1^{-1}$, which is absorbed into constants depending on $\gamma$. Since $\ell$ runs over all positive scales, the collection of genuine adapted rectangles is unchanged. The normalization ensures that the first dyadic coordinate changes at every generation. The associated dilation is
\begin{align}\label{eq:monomial-dilation}
        \delta^\gamma_\lambda(x_1,\ldots,x_d)
        =(\lambda^{\alpha_1}x_1,\ldots,\lambda^{\alpha_d}x_d).
\end{align}
\begin{introdefinition}\label{def:gamma-cubes}
For a monomial curve $\gamma$ as in \eqref{eq:monomial-curve}, a $\gamma$-cube is a translate of $\delta_\ell^\gamma([0,1]^d)$ for some $\ell>0$. Equivalently, its side-lengths are $(\ell^{\alpha_1},\ldots,\ell^{\alpha_d})$ in the coordinate frame of the monomial curve. We write $\ell(Q)=\ell$ and denote by $\Qc_\gamma$ the corresponding collection of curve-adapted cubes.
\end{introdefinition}

\begin{remark}
Following \cite[Definition~1.2 and Remark~1.4]{CladekOu2018}, a dyadic approximation of the $\gamma$-cubes at scale $2^k$ has coordinate side-lengths $2^{\lfloor k\alpha_j\rfloor}$. The bounds
\begin{align*}
2^{-1}2^{k\alpha_j}<2^{\lfloor k\alpha_j\rfloor}\le 2^{k\alpha_j}
\end{align*}
show that these boxes are uniformly comparable with genuine $\gamma$-cubes. For nonintegral exponents, dyadic approximation boxes need not be genuine $\gamma$-cubes. A fixed $1/3$-shift may fail nesting, so \Lemref{lem:corrected-grid} uses alternating shifts with nested grids whose central thirds partition space.
\end{remark}

For the local transform we take a $C^d$ curve $\gamma:[-1,1]\to\R^d$ with $\gamma(0)=0$ and assume the uniform torsion condition
\begin{align}\label{eq:torsion}
\inf_{-1\le t\le1}\left|\det\bigl(\gamma'(t),\gamma''(t),\ldots,\gamma^{(d)}(t)\bigr)\right|>0.
\end{align}
We use the local Hilbert transform
\begin{align}\label{eq:local-H}
        \Hgamloc f(x)\coloneqq \pv\int_{-1}^{1}f(x-\gamma(t))\,\frac{dt}{t} .
\end{align}
Its coordinatewise tensor extension is denoted by
\begin{align}\label{eq:vector-extension-Hloc}
        \Hvecloc\vecf=(H_{\gamma,\mathrm{loc}}f_1,\ldots,H_{\gamma,\mathrm{loc}}f_n).
\end{align}
The determinant in \eqref{eq:torsion} has a constant sign, including negative orientation. For the full-interval theorem, the closed-interval bound is essential by \Propref{prop:endpoint-flat}.
Since the Taylor vectors $\gamma'(0),\ldots,\gamma^{(d)}(0)$ form a basis, we apply Gram--Schmidt in that order to obtain an orthonormal frame $e_1,\ldots,e_d$, exactly as in \cite[Definition~1.5]{CladekOu2018}. We use the resulting fixed orthogonal coordinate change. Since $e_j$ is orthogonal to $\gamma^{(k)}(0)$ for $k<j$, the $j$th coordinate of $\gamma(t)$ satisfies $|\langle\gamma(t),e_j\rangle|\le C_\gamma|t|^j$ for $|t|\le1$. The curve-adapted cubes are translates of anisotropic dilates of $[0,1]^d$ in this orthonormal frame, with side-lengths $(\ell,\ell^2,\ldots,\ell^d)$. We use the notation $\Qc_\gamma$ for this local cube basis.

We write $\Kc=\mathscr D_\star$ for the corrected dyadic grid of \Lemref{lem:corrected-grid}, constructed from \cite[p.~422]{CladekOu2018}, and $\Qc_\gamma$ for all translated adapted boxes. The lifting uses $\Kc$, while the matrix characteristic uses $\Qc_\gamma$ to cover every metric ball required by \Extref{KakaroumpasNguyenVardakis2026}{Theorem~1.3} \cite{KakaroumpasNguyenVardakis2026}.

A family $\Sc$ of measurable curve-adapted boxes is $\delta$-sparse if there exist measurable sets $E_Q\subseteq Q$, $Q\in\Sc$, such that
\begin{align}\label{eq:sparse-definition-intro}
        |E_Q|\ge\delta |Q|,
        \qquad E_Q\cap E_{Q'}=\varnothing\quad(Q\ne Q').
\end{align}
For $1\le r<\infty$, set
\begin{align}\label{eq:avg}
        \langle h\rangle_{r,Q}
        \coloneqq \left(\frac1{|Q|}\int_Q |h(x)|^r\,dx\right)^{1/r},
\end{align}
with the usual essential supremum when $r=\infty$. For $1\le r,s\le\infty$, the scalar sparse form is
\begin{align}\label{eq:scalar-sparse}
        \Lambda_{\Sc,r,s}(f,g)
        \coloneqq \sum_{Q\in\Sc}|Q|\langle f\rangle_{r,Q}\langle g\rangle_{s,Q}.
\end{align}

The convex body sparse form used in the main results is defined as follows. It is stated for $\C^n$-valued functions, since this is the value space needed for matrix weights.

\begin{introdefinition}\label{def:convex-average-intro}
For either a genuine $\gamma$-cube or a dyadic approximation box $Q$, and for $\vecf\in L^r(Q,\C^n)$, $1\le r\le\infty$, define
\begin{align}\label{eq:convex-average}
        \llangle \vecf\rrangle_{r,Q}
        \coloneqq \overline{\left\{\frac1{|Q|}\int_Q \vecf(x)\varphi(x)\,dx:
        \langle \varphi\rangle_{r',Q}\le1\right\}}.
\end{align}
The closure guarantees compactness also when $r=\infty$, since the admissible scalar functions then satisfy $|Q|^{-1}\int_Q|\varphi|\le1$. The support function of this convex body is $h_{\llangle\vecf\rrangle_{r,Q}}(e)=\langle|\langle\vecf,e\rangle_{\C^n}|\rangle_{r,Q}$ for every $1\le r\le\infty$.
If $K,L\subseteq\C^n$ are compact symmetric convex sets, write
\begin{align}\label{eq:circ}
        K\circ L\coloneqq \sup_{v\in K,\,w\in L}|\langle v,w\rangle_{\C^n}|.
\end{align}
For a sparse family $\Sc$ of curve-adapted boxes, including the dyadic approximations in $\Kc$, define the bilinear convex body sparse form
\begin{align}\label{eq:convex-form}
        \Lambda^{\CB}_{\Sc,r,s}(\vecf,\vecg)
        \coloneqq \sum_{Q\in\Sc}|Q|\,\llangle \vecf\rrangle_{r,Q}\circ \llangle \vecg\rrangle_{s,Q}.
\end{align}
\end{introdefinition}

Let $\Oc_d$ be the closed polygon in the $(x,y)$-plane with vertices
\begin{align}\label{eq:trapezoid}
        (0,0),\quad (1,1),\quad
        \left(\frac{d^2-d+2}{d^2+d},\frac{d-1}{d+1}\right),\quad
        \left(\frac{2}{d+1},\frac{2d-2}{d^2+d}\right).
\end{align}
We denote by $\inte(\Oc_d)$ the interior of this polygon and by $\overline{\Oc_d}$ its closure. The main sparse estimates are stated for points in $\inte(\Oc_d)$, and the sharpness statement is formulated outside $\overline{\Oc_d}$.  \Figref{fig:Omega-d} records the geometry of the region.
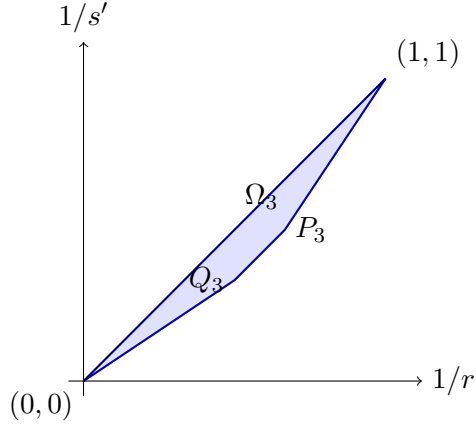
\begin{figure}[ht]
\centering
\begin{tikzpicture}[scale=4.0]
\draw[->] (-0.05,0) -- (1.12,0) node[right] {$1/r$};
\draw[->] (0,-0.05) -- (0,1.12) node[above] {$1/s'$};
\coordinate (O) at (0,0);
\coordinate (I) at (1,1);
\coordinate (P) at (0.666667,0.5);
\coordinate (Qv) at (0.5,0.333333);
\fill[blue!12] (O)--(Qv)--(P)--(I)--cycle;
\draw[blue!65!black, thick] (O)--(Qv)--(P)--(I)--cycle;
\draw[dashed] (0,0)--(1,1);
\node[below left] at (O) {$(0,0)$};
\node[above right] at (I) {$(1,1)$};
\node[right] at (P) {$P_3$};
\node[left] at (Qv) {$Q_3$};
\node at (0.59,0.61) {$\Omega_3$};
\end{tikzpicture}
\caption{The sparse region illustrated at $d=3$. The vertices are $P_3=(2/3,1/2)$ and $Q_3=(1/2,1/3)$. The general $d$-dependent coordinates are given in \eqref{eq:trapezoid}. At $d=2$ the two nontrivial vertices coincide.}
\label{fig:Omega-d}
\end{figure}
Cladek--Ou \cite{CladekOu2018} proved the scalar sparse theorem for monomial curves by decomposing $H_\gamma$ into single-scale pieces, proving translation-continuity $L^p$-improving estimates for those pieces, and iterating a stopping-time estimate over dyadic $\gamma$-cubes. The theorem below is the scalar line which we shall lift to a convex body form.

\begin{alphatheorem}[{\Extref{CladekOu2018}{Theorem~1} of \cite{CladekOu2018}}]\label{thm:A}
Let $\gamma$ be a monomial curve in $\R^d$, $d\ge2$. Suppose that $(1/r,1/s')\in\inte(\Oc_d)$. If $f,g$ are compactly supported bounded scalar functions on $\R^d$, then there exists a $\delta_\gamma$-sparse collection $\Sc\subseteq\Qc_\gamma$ such that
\begin{align}\label{eq:theoremA}
        |\langle H_\gamma f,g\rangle|
        \le C_{d,\gamma,r,s}\Lambda_{\Sc,r,s}(f,g).
\end{align}
\end{alphatheorem}

Cladek--Ou \cite[Section~1.2, p.~419]{CladekOu2018} formulate the torsion condition on $(-1,1)$, whereas their displayed local integral extends to the endpoints $-1$ and $1$. Under that literal open-interval hypothesis, \Propref{prop:endpoint-flat} supplies counterexamples for the full integral. Our positive statement assumes \eqref{eq:torsion} on the closed interval. Its proof verifies the scale-uniform hypotheses used in \cite[Section~3, Lemma~3.2]{CladekOu2018}, rather than invoking the open-interval formulation without qualification. Restriction to $[-a,a]$, $0<a<1$, recovers a local conclusion from interior nonvanishing torsion.

\begin{alphatheorem}[{Restricted form of \Extref{CladekOu2018}{Theorem~2} in \cite{CladekOu2018}}]\label{thm:B}
Let $\gamma:[-1,1]\to\R^d$ be a $C^d$ curve with $\gamma(0)=0$ satisfying \eqref{eq:torsion}, $d\ge2$. Suppose that $(1/r,1/s')\in\inte(\Oc_d)$. If $f,g$ are compactly supported bounded scalar functions on $\R^d$, then there exists a $\delta_\gamma$-sparse collection $\Sc\subseteq\Qc_\gamma$ such that
\begin{align}\label{eq:theoremB}
        |\langle H_{\gamma,\mathrm{loc}} f,g\rangle|
        \le C_{d,\gamma,r,s}\Lambda_{\Sc,r,s}(f,g),
\end{align}
where $H_{\gamma,\mathrm{loc}}$ is the local operator in \eqref{eq:local-H}.
\end{alphatheorem}

We establish the single-scale sharpness obstruction by three separated-support Knapp tests. The positive sparse bounds are obtained through a direction-preserving lifting argument on anisotropic grids and the matrix sparse embedding of \Extref{KakaroumpasNguyenVardakis2026}{Theorem~1.3} \cite{KakaroumpasNguyenVardakis2026}. In view of the general scalar-to-convex-body lifting theorem \Extref{LaukkarinenLorist2026}{Theorems~2.8 and~3.4} \cite{LaukkarinenLorist2026}, the essential issues here are the anisotropic dyadic reduction and control of the matrix characteristic over all translated adapted cubes. This leads to two questions.

\begin{question}\label{question:convex}
Let $H_\gamma$ be one of the curve Hilbert transforms in \Thmref{thm:A} or \Thmref{thm:B}. Does $\Hvec$ satisfy a bilinear convex body sparse form bound in the same exponent region as the scalar theorem?
\end{question}

\begin{question}\label{question:matrix}
Can such a convex body sparse form be inserted into a limited-range matrix weighted embedding theorem to give quantitative $L^p(W)$ boundedness for $\Hvec$?
\end{question}

The dyadic lifting yields genuine adapted boxes, while endpoint contact supplies an independent obstruction. For related frameworks see \cite{MetricSparse2022,MatrixSurvey2025}.

\subsection{Main results}\label{subsec:main-results}
\vspace{0.1em}\noindent
The two positive theorems cover $\inte(\Oc_d)$ and the necessity statement excludes exponents outside $\overline{\Oc_d}$ in the stated range. The boundary is not settled. See \Defref{def:convex-average-intro} for the convex body form.

\begin{theorem}\label{thm:monomial-main}
Let $d\ge2$, $n\ge1$, and let $\gamma$ be a monomial curve in $\R^d$ as in \eqref{eq:monomial-curve}. Suppose that $(1/r,1/s')\in\inte(\Oc_d)$. If $\vecf,\vecg:\R^d\to\C^n$ are compactly supported bounded functions, then there exist $\delta_\gamma>0$ and one $\delta_\gamma$-sparse family $\Sc\subseteq\Qc_\gamma$ of genuine curve-adapted cubes such that
\begin{align}\label{eq:main-monomial}
        |\langle \Hvec\vecf,\vecg\rangle|
        \le C_{d,\gamma,r,s,n}\Lambda^{\CB}_{\Sc,r,s}(\vecf,\vecg).
\end{align}
The family depends on the tested functions and has a positive sparsity constant determined only by the adapted geometry.
\end{theorem}

For local torsion curves, the origin Taylor frame replaces global monomial dilations.

\begin{theorem}\label{thm:torsion-main}
Let $d\ge2$, $n\ge1$, and let $\gamma:[-1,1]\to\R^d$ be a $C^d$ curve with $\gamma(0)=0$ satisfying the uniform torsion condition \eqref{eq:torsion}. Suppose that $(1/r,1/s')\in\inte(\Oc_d)$. If $\vecf,\vecg:\R^d\to\C^n$ are compactly supported bounded functions, then there exist $\delta_\gamma>0$ and one $\delta_\gamma$-sparse family $\Sc\subseteq\Qc_\gamma$ of genuine curve-adapted cubes such that
\begin{align}\label{eq:main-torsion}
        |\langle \Hvecloc\vecf,\vecg\rangle|
        \le C_{d,\gamma,r,s,n}\Lambda^{\CB}_{\Sc,r,s}(\vecf,\vecg).
\end{align}
The family depends on the tested functions and has a positive sparsity constant determined only by the adapted geometry.
\end{theorem}

The local conclusion also applies to $\pv\int_{-a}^{a}f(x-\gamma(t))\,dt/t$ when $0<a<1$ and the torsion of $\gamma$ is nonzero on $(-1,1)$. Indeed, the rescaled curve $\gamma_a(t)\coloneqq\gamma(at)$ has uniformly nonvanishing torsion on $[-1,1]$ and determines the same orthonormal Taylor frame. This provides the usual genuinely local formulation without imposing an endpoint condition at $t=\pm1$.

The matrix bound uses $[W]_{p,r,s}$ from \eqref{eq:matrix-characteristic}, measured over all translated adapted boxes.

\begin{corollary}\label{cor:matrix}
Let $1\le r<p<s'<\infty$, $n\ge1$, and let $W:\R^d\to\C^{n\times n}$ be a matrix weight satisfying $[W]_{p,r,s}<\infty$. Assume that $\gamma$ is either a monomial curve in $\R^d$ or a $C^d$ curve on $[-1,1]$ with $\gamma(0)=0$ satisfying \eqref{eq:torsion}. If $(1/r,1/s')\in\inte(\Oc_d)$, then the following corresponding estimate holds:
\begin{align}\label{eq:matrix-main}
\begin{cases}
        \|\Hvec\|_{L^p(W)\to L^p(W)}
        \le C_{d,\gamma,p,r,s,n}
        [W]_{p,r,s}^{\frac{s'-p}{s'(p-r)}+\frac1{p'}}, & \gamma\hbox{ monomial},\\[0.35em]
        \|\Hvecloc\|_{L^p(W)\to L^p(W)}
        \le C_{d,\gamma,p,r,s,n}
        [W]_{p,r,s}^{\frac{s'-p}{s'(p-r)}+\frac1{p'}}, & \gamma\hbox{ torsion}.
\end{cases}
\end{align}
The constant is independent of the matrix weight $W$.
\end{corollary}

Finally, scalar sharpness also excludes any fixed finite family of convex body sparse forms.

\begin{corollary}\label{cor:sharpness}
Let $\gamma$ be either a monomial curve in $\R^d$ or a $C^d$ curve on $[-1,1]$ with $\gamma(0)=0$ and torsion uniformly nonvanishing on $[-1,1]$, and let $n\ge1$. Suppose that $1<r,s<\infty$, $1/r+1/s>1$, and $(1/r,1/s')\in\R^2\setminus\overline{\Oc_d}$. For every fixed integer $J\ge1$ and sparsity parameter $\delta\in(0,1)$, there is no constant $C$ such that every pair of compactly supported bounded $\C^n$-valued functions admits $J$ $\delta$-sparse families $\Sc_1,\ldots,\Sc_J$ of genuine curve-adapted cubes or their comparable dyadic approximation boxes satisfying
\begin{align}\label{eq:forbidden-finite-sparse}
|\langle\overrightarrow H\vecf,\vecg\rangle|
\le C\sum_{j=1}^J\Lambda^{\CB}_{\Sc_j,r,s}(\vecf,\vecg),
\end{align}
where $\overrightarrow H$ denotes the full monomial transform or the local torsion transform, respectively.
\end{corollary}

\begin{remark}\label{rem:main-results-comparison}
The convex body bounds retain the scalar sharp exponent range up to its boundary. The lifting and matrix embedding use \cite{LaukkarinenLorist2026,KakaroumpasNguyenVardakis2026}. The endpoint-contact obstruction in \Propref{prop:endpoint-flat} occurs inside the scalar region for suitable curves, whereas \Corref{cor:sharpness} concerns exponents outside $\overline{\Oc_d}$ for uniformly nondegenerate curves. No optimality of the matrix-weight exponent is asserted.
\end{remark}

\subsection{Organization}\label{subsec:organization}
\Secref{sec:convex-proof} states the auxiliary results and proves the main theorems. \Secref{sec:auxiliary-estimates} supplies the auxiliary proofs, including the adapted metric construction.

\subsection{Notation}\label{subsec:notation}
\vspace{0.1em}\noindent
We use the following conventions throughout the paper.
\begin{itemize}
\item The vector-valued pairing is $\langle\vecf,\vecg\rangle\coloneqq\int_{\R^d}\langle\vecf(x),\vecg(x)\rangle_{\C^n}\,dx$.

\item The family $\Qc_\gamma$ denotes the curve-adapted cubes. The symbol $\Kc=\mathscr D_\star$ denotes the fixed shifted dyadic curve grid used for the lifting. It is the dyadic model of $\Qc_\gamma$ used in \cite[pp.~422--423]{CladekOu2018}. The asymmetric sparse family is a subfamily of $\Kc$. The dyadic boxes in $\Kc$ need not be genuine members of $\Qc_\gamma$. Sparse averages are defined on both collections, whereas the matrix characteristic is measured on the full basis $\Qc_\gamma$.

\item The scalar average, scalar sparse form, convex body average and convex body sparse form are defined in \eqref{eq:avg}, \eqref{eq:scalar-sparse}, \eqref{eq:convex-average} and \eqref{eq:convex-form}. The stopping generations and their convex body normalizations are defined in \Secref{subsec:adapted-lifting-proof}.

\item A matrix weight is a locally integrable function $W:\R^d\to\C^{n\times n}$ such that $W(x)$ is self-adjoint and positive definite for a.e. $x$. For $1<p<\infty$ we set
\begin{align}\label{eq:matrix-space}
        \|\vecf\|_{L^p(W)}
        \coloneqq \left(\int_{\R^d}|W(x)^{1/p}\vecf(x)|^p\,dx\right)^{1/p}.
\end{align}

\item For a matrix weight $U$ and $1<\alpha\le\beta<\infty$, define the two-exponent matrix characteristic over all curve-adapted boxes by
\begin{align}\label{eq:matrix-Aab}
[U]_{A_{\alpha,\beta}}
\coloneqq \sup_{Q\in\Qc_\gamma}\frac1{|Q|}\int_Q
\left(\frac1{|Q|}\int_Q\|U(x)^{1/\beta}U(y)^{-1/\beta}\|^{\alpha'}\,dy\right)^{\beta/\alpha'}dx .
\end{align}
When $1\le r<p<s'<\infty$, we use the matrix characteristic
\begin{align}\label{eq:matrix-characteristic}
        [W]_{p,r,s}
        \coloneqq \left[W^{\frac{s'}{s'-p}}\right]_{A_{\left(\frac{pr}{p-r}\right)',\frac{ps'}{s'-p}}}.
\end{align}
The finiteness hypothesis is understood in the sense of the matrix-weight class in \cite{KakaroumpasNguyenVardakis2026}, in particular the power $W^{s'/(s'-p)}$ is locally integrable.

\item We write $X\lesssim Y$ if $X\le CY$ for an admissible constant $C$, and $X\eqsim Y$ if both $X\lesssim Y$ and $Y\lesssim X$. Constants may depend on the curve, the dimension and the exponents, but not on the functions or the matrix weight.
\end{itemize}

\section{Proofs of the main results}\label{sec:convex-proof}

We first state the geometric tools used in both main theorems. Their proofs are collected in \Secref{sec:auxiliary-estimates}. The alternating shifts below provide the nested boxes and the exact central-third partition needed in the scalar stopping argument.

\begin{lemma}[Nested curve-adapted thirds]\label{lem:corrected-grid}
For monomials normalize $\alpha_1=1$, and for the local Taylor frame take $\alpha_j=j$. Set $\nu_j(k)\coloneqq\lfloor k\alpha_j\rfloor$. For $a\in\{0,1,2\}$ define
\begin{align}\label{eq:alternating-third-intervals}
I^{a}_{\ell,m}\coloneqq 2^\ell
\bigl[m+(-1)^\ell a/3,\,m+1+(-1)^\ell a/3\bigr),
\qquad \ell,m\in\Z.
\end{align}
For $\mathbf a=(a_1,\ldots,a_d)\in\{0,1,2\}^d$ define
\begin{align}\label{eq:alternating-adapted-grids}
\mathscr D_{\mathbf a}(k)
\coloneqq\left\{\prod_{j=1}^d I^{a_j}_{\nu_j(k),m_j}:
(m_1,\ldots,m_d)\in\Z^d\right\},\qquad
\mathscr D_{\mathbf a}\coloneqq\bigcup_{k\in\Z}\mathscr D_{\mathbf a}(k).
\end{align}
The construction has four properties. First, each $\mathscr D_{\mathbf a}(k)$ partitions $\R^d$, and in a fixed grid
\begin{align*}
Q\cap P\ne\varnothing&\Longrightarrow Q\subseteq P\ \text{or}\ P\subseteq Q,\\
1\le\frac{|Q^{(1)}|}{|Q|}&\le 2^{\sum_{j=1}^d\lceil\alpha_j\rceil},
\qquad Q\in\mathscr D_{\mathbf a}(k).
\end{align*}
Second, the boxes containing almost every $x$ shrink to $x$ as $k\to-\infty$. Third, the central-third covering is exact:
\begin{align*}
\sum_{\mathbf a\in\{0,1,2\}^d}\ \sum_{Q\in\mathscr D_{\mathbf a}(k)}
\mathbf1_{\frac13Q}(x)=1\quad\text{for a.e. }x\in\R^d.
\end{align*}
Finally, for $Q\in\mathscr D_{\mathbf a}(k)$,
\begin{align*}
2^{-1}2^{k\alpha_j}<\operatorname{side}_j(Q)\le2^{k\alpha_j},\qquad 1\le j\le d.
\end{align*}
Thus every such box is comparable with a genuine $\gamma$-cube.
\end{lemma}

For the lifting we fix $\mathscr D_\star\coloneqq\mathscr D_{(1,\ldots,1)}$. Its sufficiently large boxes contain any prescribed compact set. For $Q\in\mathscr D_\star(k)$ let $Q^+$ be the genuine adapted box with center $c_Q$ and coordinate side-lengths $(A_\gamma2^k)^{\alpha_i}$, where $A_\gamma\ge8$ is fixed. We shall use
\begin{align*}
Q\subseteq Q^+,\qquad |Q^+|\eqsim_\gamma|Q|,\qquad
Q\subseteq P\Longrightarrow Q^+\subseteq P^+.
\end{align*}
These inclusions are established at the start of the proof of \Propref{prop:adapted-lifting}. Only the scalar sparse bounds of \Thmref{thm:A} and \Thmref{thm:B} are required.

The next proposition turns scalar sparse domination into one convex body sparse form. We adapt \cite[Theorem~3.4]{LaukkarinenLorist2026} through an asymmetric dyadic bound, followed by enlargement to genuine adapted boxes.

\begin{proposition}[Adapted scalar-to-convex-body lifting]\label{prop:adapted-lifting}
Let $\Lambda$ be a continuous scalar sesquilinear form on $L^p(\R^d)\times L^{p'}(\R^d)$, where
\begin{align*}
1\le r,s<\infty,\qquad r<p<s',\qquad
\theta\coloneqq\frac1r+\frac1s-1>0.
\end{align*}
For vector test functions we use the coordinatewise extension
\begin{align*}
\Lambda(\vecf,\vecg)\coloneqq\sum_{j=1}^n\Lambda(f_j,g_j).
\end{align*}
Assume that every scalar test pair satisfies a uniform $(r,s)$-sparse bound over genuine adapted boxes. We first obtain an asymmetric bound on the fixed grid. For every bounded compactly supported pair $\vecf,\vecg:\R^d\to\C^n$, there is an $\eta_\gamma$-sparse family $\Sc_\star\subseteq\mathscr D_\star$ such that
\begin{align}\label{eq:adapted-asymmetric-lifting}
|\Lambda(\vecf,\vecg)|
&\le C_{\gamma,r,s,n}\sum_{Q\in\Sc_{\star}}|Q|\,
\llangle\vecf\rrangle_{r,Q^+}\circ\llangle\vecg\rrangle_{s,Q}.
\end{align}
Enlarging the selected boxes then produces a single $\delta_\gamma$-sparse family $\Sc\subseteq\Qc_\gamma$ of genuine adapted boxes satisfying
\begin{align}\label{eq:adapted-one-genuine}
|\Lambda(\vecf,\vecg)|\le C_{\gamma,r,s,n}
\sum_{R\in\Sc}|R|\,
\llangle\vecf\rrangle_{r,R}\circ\llangle\vecg\rrangle_{s,R}.
\end{align}
The constants depend on the geometry and scalar sparse bound. The dyadic family in \eqref{eq:adapted-asymmetric-lifting} is used for matrix embedding.
\end{proposition}

For nonzero planar endpoint curvature the contact order is $2$, and the resulting threshold $\theta>1/3$ lies outside $\inte(\Oc_2)$. Higher contact order moves the obstruction inside this region, as follows.

\begin{proposition}[Endpoint contact and sparse obstruction]\label{prop:endpoint-flat}
Let $\beta>1$ and $\gamma\in C^2([-1,1],\R^2)$. Assume
\begin{align*}
\gamma(0)&=0, & \det(\gamma'(0),\gamma''(0))&\ne0,\\
\gamma'(1)&\ne0, & \operatorname{dist}\bigl(\gamma(1),\gamma([-1,0])\bigr)&>0.
\end{align*}
Let $|\nu|=1$ and $\nu\perp\gamma'(1)$. The endpoint has contact order at least $\beta$ if, for some $C_\gamma,u_0>0$,
\begin{align}\label{eq:endpoint-contact-condition-r24}
\bigl|\langle\gamma(1-u)-\gamma(1),\nu\rangle\bigr|\le C_\gamma u^\beta,
\qquad 0<u<u_0.
\end{align}
Fix $1<r,s<\infty$ with
\begin{align}\label{eq:endpoint-threshold-r37}
\theta\coloneqq\frac1r+\frac1s-1>\frac1{\beta+1}.
\end{align}
Then the full-interval $H_{\gamma,\mathrm{loc}}$ has no uniform $(r,s)$-sparse estimate over genuine adapted boxes, even with any prescribed finite number of sparse families. More precisely, for each $J\in\mathbb N$, $0<\eta<1$ and $C>0$, some bounded compactly supported $f,g$ satisfy
\begin{align}\label{eq:endpoint-no-finite-sparse-r37}
|\langle H_{\gamma,\mathrm{loc}}f,g\rangle|
>C\sum_{j=1}^J\Lambda_{\Sc_j,r,s}(f,g)
\end{align}
for every choice of $\eta$-sparse families $\Sc_1,\ldots,\Sc_J\subseteq\Qc_\gamma$. The terminal operator satisfies
\begin{align}\label{eq:endpoint-terminal-unbounded-r37}
T_{\gamma,a}f(x)\coloneqq\int_a^1f(x-\gamma(t))\,\frac{dt}{t},\qquad
\|T_{\gamma,a}\|_{L^r(\R^2)\to L^{s'}(\R^2)}=\infty
\quad(0<a<1).
\end{align}

For every integer $m\ge3$, the polynomial curve
\begin{align}\label{eq:endpoint-polynomial-curve}
\gamma_m(t)\coloneqq\left(t,\frac{(1-t)^m-1}{m(m-1)}\right),\qquad -1\le t\le1,
\end{align}
satisfies $\gamma_m(0)=0$ and
\begin{align}\label{eq:endpoint-flat-torsion}
\det(\gamma_m'(t),\gamma_m''(t))=(1-t)^{m-2}>0\quad(-1<t<1),\qquad
\det(\gamma_m'(1),\gamma_m''(1))=0.
\end{align}
For $\gamma_m$, the obstruction threshold and a specific interior example are
\begin{align*}
\theta&>\frac1{m+1},\\
m=3,\qquad r=s&=\frac{20}{13},\qquad
\theta=\frac3{10}>\frac14,\qquad
\Bigl(\frac1r,\frac1{s'}\Bigr)=\Bigl(\frac{13}{20},\frac7{20}\Bigr)\in\inte(\Oc_2).
\end{align*}
In fact, for every $(1/r,1/s')\in\inte(\Oc_2)$ a sufficiently large $m$ gives such a counterexample with torsion positive on $(-1,1)$. An infinitely flat $C^\infty$ example yields the same failure for every $\theta>0$.
\end{proposition}

The sharpness proof uses the separated-support test below, proved in \Secref{sec:auxiliary-estimates}.
\begin{lemma}[Separated-support Knapp tests]\label{lem:separated-knapp-r35}
Let $d\ge2$, $\eta\in C^d([1/2,1],\R^d)$, and suppose, for some $t_0\in(1/2,1)$ and nonzero real linear functional $L$, that
\begin{align*}
\det\bigl(\eta'(t_0),\ldots,\eta^{(d)}(t_0)\bigr)\ne0,
\qquad L(\eta(t_0))>0.
\end{align*}
Fix $1<r,s<\infty$ with $1/r+1/s>1$ and $(1/r,1/s')\notin\overline{\Oc_d}$. Then some fixed $I\Subset(1/2,1)$ and bounded sets $E,F\subseteq\R^d$ satisfy
\begin{align*}
\operatorname{dist}(E,F)>0,\qquad
\inf_{z\in F-E}L(z)>0.
\end{align*}
There exist nonnegative bounded $u_k,v_k$ supported in $E,F$ respectively, such that
\begin{align}\label{eq:knapp-unnormalized-r35}
\|u_k\|_{L^r}=\|v_k\|_{L^s}=1,\qquad
\int_I\int_{\R^d}u_k(z-\eta(t))v_k(z)\,dz\,\frac{dt}{t}\longrightarrow\infty.
\end{align}
\end{lemma}

\begin{remark}\label{rem:endpoint-literature-r32}
In \cite[Section~1.2 and Theorem~2]{CladekOu2018}, the printed torsion condition concerns $-1<t<1$, while the displayed integral extends over $[-1,1]$. The distinction matters for the full-interval conclusion: the curves in \Propref{prop:endpoint-flat} satisfy the interior condition but violate a sparse bound for suitable exponents in $\inte(\Oc_2)$. In contrast, \Thmref{thm:torsion-main} assumes
\begin{align*}
\inf_{-1\le t\le1}\bigl|\det(\gamma'(t),\ldots,\gamma^{(d)}(t))\bigr|>0.
\end{align*}
The terminal average $T_{\gamma,a}$ already exhibits the obstruction. For torsion nonzero only on $(-1,1)$, the positive conclusion applies after restriction to $[-a,a]$, $0<a<1$.
\end{remark}

The asymmetric family also yields the matrix weighted bound. Its proof will compare the boxes $Q\subseteq Q^+$ with the metric neighborhoods $\lambda Q$ required by the matrix embedding theorem.
\begin{proposition}\label{prop:B}
Let $W:\R^d\to\C^{n\times n}$ be a matrix weight, and assume
\begin{align*}
1<p<\infty,\qquad 1\le r<p<s'<\infty,\qquad [W]_{p,r,s}<\infty.
\end{align*}
For every sparse $\Sc\subseteq\mathscr D_\star$ and bounded compactly supported $\vecF,\vecG:\R^d\to\C^n$, we claim
\begin{align}\label{eq:matrix-embedding}
\sum_{Q\in\Sc}|Q|\,
\llangle W^{-1/p}\vecF\rrangle_{r,Q^+}\circ
\llangle W^{1/p}\vecG\rrangle_{s,Q}
\le C[W]_{p,r,s}^{\frac{s'-p}{s'(p-r)}+\frac1{p'}}
\|\vecF\|_{L^p}\|\vecG\|_{L^{p'}}.
\end{align}
Here $C$ depends only on the adapted geometry, $n,p,r,s$ and the sparsity parameter.
\end{proposition}

To pass from \Propref{prop:B} to the operator norm in \Corref{cor:matrix}, we use the weighted duality identity
\begin{align}\label{eq:duality}
        \|\vech\|_{L^p(W)}
        =\sup_{\|\vecG\|_{L^{p'}(W^{-p'/p})}=1}
        \left|\int_{\R^d}\langle\vech(x),\vecG(x)\rangle_{\C^n}\,dx\right|.
\end{align}

\medskip

\begin{proof}[Proof of \Thmref{thm:monomial-main}]
Fix $r<p<s'$. First, the scalar sparse bound in \Thmref{thm:A} gives continuity:
\begin{align}\label{eq:adapted-continuity-r33}
\Lambda_{\Sc,r,s}(f,g)
&\le\delta_{\gamma}^{-1}\int_{\R^d}M_r^\gamma f(x)M_s^\gamma g(x)\,dx
\le C_{\gamma,p,r,s}\|f\|_{L^p}\|g\|_{L^{p'}},
\end{align}
Thus $\Lambda\in\mathcal B(L^p\times L^{p'},\C)$. By \Propref{prop:adapted-lifting} we obtain one genuine adapted sparse family:
\begin{align*}
|\langle\Hvec\vecf,\vecg\rangle|
\lesssim_{\gamma,r,s,n}\sum_{Q\in\Sc}|Q|\,
\llangle\vecf\rrangle_{r,Q}\circ\llangle\vecg\rrangle_{s,Q}.
\end{align*}
This is \eqref{eq:main-monomial}. The asymmetric bound \eqref{eq:adapted-asymmetric-lifting} supplies the dyadic estimate for \Corref{cor:matrix}.
\end{proof}

\begin{proof}[Proof of \Thmref{thm:torsion-main}]
\noindent\textit{Uniform geometry.} We first control every annular scale $0<\lambda\le1$, including the terminal annuli. In the fixed Taylor frame write $\gamma_i(t)\coloneqq\langle\gamma(t),e_i\rangle$ and rescale:
\begin{align}\label{eq:torsion-rescaling-r34}
\Gamma_\lambda(u)\coloneqq\delta_{\lambda}^{-1}\gamma(\lambda u),\qquad
\delta_{\lambda}(x_1,\ldots,x_d)\coloneqq(\lambda x_1,\ldots,\lambda^d x_d),
\qquad 0<\lambda\le1,\quad |u|\le1.
\end{align}
For $0\le j\le d$ and $1\le i\le d$, we have
\begin{align}\label{eq:torsion-derivatives-r34}
\partial_u^j\Gamma_{\lambda,i}(u)
=\lambda^{j-i}\gamma_i^{(j)}(\lambda u).
\end{align}
The Taylor-frame cancellations are $\gamma_i^{(j)}(0)=0$ for $0\le j<i$. For these indices, Taylor's integral formula yields
\begin{align}\label{eq:torsion-uniform-taylor-r34}
\lambda^{j-i}\gamma_i^{(j)}(\lambda u)
=\frac{u^{i-j}}{(i-j-1)!}\int_0^1(1-v)^{i-j-1}\gamma_i^{(i)}(v\lambda u)\,dv.
\end{align}
For $j\ge i$ the right side of \eqref{eq:torsion-derivatives-r34} is bounded by $\|\gamma_i^{(j)}\|_{L^\infty}$, since $\lambda^{j-i}\le1$. Consequently the rescaled curves satisfy
\begin{align}\label{eq:torsion-compact-family-r34}
\sup_{0<\lambda\le1}\|\Gamma_\lambda\|_{C^d([-1,1])}&\le C_\gamma,\\
\inf_{0<\lambda\le1}\inf_{|u|\le1}
\bigl|\det(\Gamma_\lambda'(u),\ldots,\Gamma_\lambda^{(d)}(u))\bigr|
&=\inf_{|t|\le1}\bigl|\det(\gamma'(t),\ldots,\gamma^{(d)}(t))\bigr|>0.\notag
\end{align}
The determinant scaling is exact:
\begin{align*}
\det(\Gamma_\lambda',\ldots,\Gamma_\lambda^{(d)})
&=\lambda^{\sum_{j=1}^d j-\sum_{i=1}^d i}\det(\gamma'(\lambda u),\ldots,\gamma^{(d)}(\lambda u))\\
&=\det(\gamma'(\lambda u),\ldots,\gamma^{(d)}(\lambda u)).
\end{align*}
By \eqref{eq:torsion-uniform-taylor-r34}, the $C^d$ limit as $\lambda\downarrow0$ is
\begin{align}\label{eq:torsion-moment-limit-r34}
\Gamma_0(u)=\left(\gamma_1'(0)u,\frac{\gamma_2''(0)}{2!}u^2,\ldots,\frac{\gamma_d^{(d)}(0)}{d!}u^d\right).
\end{align}
\par\Needspace{3\baselineskip}\smallskip\noindent\textit{Uniform finite type.} The torsion lower bound gives a uniform derivative estimate. Let $V_\lambda(u)$ have columns $\Gamma_\lambda'(u),\ldots,\Gamma_\lambda^{(d)}(u)$, with lengths at most $M_\gamma\ge1$. The singular values give, for $|\omega|=1$,
\begin{align}\label{eq:torsion-uniform-finite-type-r36}
\max_{1\le j\le d}\bigl|\langle\omega,\Gamma_\lambda^{(j)}(u)\rangle\bigr|
&\ge\frac{1}{\sqrt d}\,\sigma_{\min}(V_\lambda(u))
\ge\frac{c_\gamma}{\sqrt d\,(\sqrt d M_\gamma)^{d-1}}>0,
\qquad 0\le\lambda\le1,\quad |u|\le1,
\end{align}
Here $c_\gamma$ is the torsion lower bound. Uniform $C^d$ continuity gives a partition into at most $N_\gamma$ intervals $J$ on each half-annulus, with a derivative order $j(J,\lambda,\omega)\in\{1,\ldots,d\}$ such that
\begin{align*}
\inf_{u\in J}|\langle\omega,\Gamma_\lambda^{(j(J,\lambda,\omega))}(u)\rangle|
\ge \frac{c_\gamma}{2\sqrt d(\sqrt dM_\gamma)^{d-1}}.
\end{align*}
The bound is independent of $\lambda,\omega$. We apply Van der Corput's lemma to amplitudes $\psi$ supported in the two half-annuli with $\|\psi\|_\infty+\operatorname{Var}_{\R}(\psi)<\infty$. Extending $\psi/u$ by zero outside its support gives
\begin{align*}
\left\|\frac{\psi}{u}\right\|_\infty
+\operatorname{Var}_{\R}\left(\frac{\psi}{u}\right)
\lesssim \|\psi\|_\infty+\operatorname{Var}_{\R}(\psi).
\end{align*}
For derivative order $j=1$, integration by parts applies since $|\langle\xi,\Gamma_\lambda'\rangle|\gtrsim_\gamma|\xi|$ and $|\langle\xi,\Gamma_\lambda''\rangle|\lesssim_\gamma|\xi|$. For $j\ge2$, the bounded-variation form of Van der Corput's lemma applies. The at most $N_\gamma$ intervals give
\begin{align}\label{eq:torsion-uniform-fourier-r36}
\sup_{0\le\lambda\le1}\biggl|\int e^{-i\langle\xi,\Gamma_\lambda(u)\rangle}\psi(u)\,\frac{du}{u}\biggr|
\le C_{\gamma}\bigl(\|\psi\|_\infty+\operatorname{Var}_{\R}(\psi)\bigr)(1+|\xi|)^{-1/d}.
\end{align}
\par\Needspace{3\baselineskip}\smallskip\noindent\textit{Quantitative Jacobian.} Fix $u_1<\cdots<u_d$ in an interval of length at most $h_\gamma$. Newton's divided-difference identity gives
\begin{align}\label{eq:torsion-vandermonde-r38}
&\det\bigl(\Gamma_\lambda'(u_1),\ldots,\Gamma_\lambda'(u_d)\bigr)\notag\\
&\quad=\prod_{1\le i<j\le d}(u_j-u_i)
\det\bigl(\Gamma_\lambda'[u_1],\Gamma_\lambda'[u_1,u_2],\ldots,\Gamma_\lambda'[u_1,\ldots,u_d]\bigr),
\end{align}
where brackets denote vector-valued divided differences. For $1\le j\le d$, the integral representation of divided differences and the compactness in \eqref{eq:torsion-compact-family-r34} yield
\begin{align}\label{eq:torsion-divided-difference-r38}
\sup_{0\le\lambda\le1}\sup_{\substack{u_1<\cdots<u_j\\u_j-u_1\le h}}
\left|\Gamma_\lambda'[u_1,\ldots,u_j]
-\frac{\Gamma_\lambda^{(j)}(u_1)}{(j-1)!}\right|\longrightarrow0
\qquad(h\downarrow0).
\end{align}
Consequently, choosing $h_\gamma>0$ sufficiently small in terms of $c_\gamma$ and the common $C^d$ modulus of the family, we obtain
\begin{align}\label{eq:torsion-configuration-r38}
\bigl|\det\bigl(\Gamma_\lambda'(u_1),\ldots,\Gamma_\lambda'(u_d)\bigr)\bigr|
&\ge \frac{c_\gamma}{2\prod_{j=0}^{d-1}j!}
\prod_{1\le i<j\le d}|u_j-u_i|,\\
&\hspace{18mm}0\le\lambda\le1,\qquad u_d-u_1\le h_\gamma.\notag
\end{align}
Both annuli can therefore be partitioned into a fixed number of intervals on which this Jacobian bound holds.
\par\Needspace{3\baselineskip}\smallskip\noindent\textit{Improving estimates.} We verify the quantitative $C^d$ hypotheses on each interval. Set
\begin{align}\label{eq:torsion-affine-density-r40}
L_\lambda(u)&\coloneqq\det(\Gamma_\lambda'(u),\ldots,\Gamma_\lambda^{(d)}(u)),\\
\omega_\lambda(u)&\coloneqq|L_\lambda(u)|^{2/[d(d+1)]},\qquad
0<c_\gamma\le|L_\lambda(u)|\le C_\gamma.\notag
\end{align}
Partition $[-1,-1/2]$ and $[1/2,1]$ into at most $N_\gamma$ intervals $J$ of length at most $h_\gamma$. By \eqref{eq:torsion-configuration-r38}, for $u_1,\ldots,u_d\in J$,
\begin{align}\label{eq:torsion-ds-jacobian-r40}
\left|\det(\Gamma_\lambda'(u_1),\ldots,\Gamma_\lambda'(u_d))\right|
&\ge c_\gamma\prod_{i<j}|u_i-u_j|\\
&\ge c'_\gamma\prod_{j=1}^d|L_\lambda(u_j)|^{1/d}
\prod_{i<j}|u_i-u_j|,\notag\\
 c_\gamma^{2/[d(d+1)]}\le\omega_\lambda(u)&\le C_\gamma^{2/[d(d+1)]}.
 \notag
\end{align}
The two hypotheses of \cite[Proposition~3.1]{DendrinosStovall2015} are now explicit in \eqref{eq:torsion-ds-jacobian-r40}: the Jacobian inequality and $\omega_\lambda\eqsim_\gamma1$ on each $J$. Thus, for measurable $E,F$ of finite measure,
\begin{align}\label{eq:torsion-rwt-r40}
\int_J\int_{\R^d}\mathbf1_E(x-\Gamma_\lambda(u))\mathbf1_F(x)\,dx\,du
&\lesssim_{\gamma,d}|E|^{2/(d+1)}
|F|^{(d^2-d+2)/(d^2+d)}.
\end{align}
The comparison $\omega_\lambda\eqsim_\gamma1$ gives the same estimate for $du$, and the adjoint obeys the reflected bound. For intervals meeting a half-annulus endpoint, apply the estimate first on compact subintervals of $\operatorname{int}J$ and pass to $J$ by monotone convergence. The constants are unchanged. In the $(1/r_0,1/s'_0)$-plane these are precisely the two nontrivial vertices of $\Oc_d$:
\begin{align}\label{eq:torsion-rwt-vertices-r40}
\left(\frac2{d+1},\frac{2(d-1)}{d(d+1)}\right),\qquad
\left(\frac{d^2-d+2}{d(d+1)},\frac{d-1}{d+1}\right).
\end{align}
The positive averages also satisfy uniform $L^1\to L^1$ and $L^\infty\to L^\infty$ bounds. Restricted weak-type interpolation at these four vertices gives, for every $(1/r_0,1/s'_0)\in\inte(\Oc_d)$,
\begin{align}\label{eq:torsion-strong-improving-r40}
\sup_{0\le\lambda\le1}
\left\|\int_J f(\cdot-\Gamma_\lambda(u))\,du\right\|_{L^{s_0'}}
\le C_{\gamma,r_0,s'_0}\|f\|_{L^{r_0}}.
\end{align}
For $\psi$ supported in the half-annuli with $\|\psi\|_\infty+\operatorname{Var}_{\R}(\psi)<\infty$, define $A_\lambda f\coloneqq\int f(\cdot-\Gamma_\lambda(u))\psi(u)\,du/u$. The improving bound follows from $|\psi(u)/u|\le2\|\psi\|_\infty$ and the fixed interval count $N_\gamma$. For $|y|\le1$, Plancherel and \eqref{eq:torsion-uniform-fourier-r36} give
\begin{align*}
\|A_\lambda-\tau_yA_\lambda\|_{L^2\to L^2}
&\le C_\gamma\bigl(\|\psi\|_\infty+\operatorname{Var}_{\R}(\psi)\bigr)
\sup_{\xi\in\R^d}\min\{2,|y||\xi|\}(1+|\xi|)^{-1/d}\\
&\lesssim_\gamma\bigl(\|\psi\|_\infty+\operatorname{Var}_{\R}(\psi)\bigr)|y|^{1/d}.
\end{align*}
Consequently,
\begin{align}\label{eq:torsion-quantitative-improving-r38}
\sup_{0\le\lambda\le1}\|A_\lambda\|_{L^{r_0}\to L^{s_0'}}&\le C_{\gamma,r_0,s_0',\psi},\\
\sup_{0\le\lambda\le1}\|A_\lambda-\tau_yA_\lambda\|_{L^2\to L^2}
&\lesssim_{\gamma,\psi}|y|^{1/d},\qquad |y|\le1.\notag
\end{align}
For a prescribed interior pair $(1/r,1/s')$, choose $(1/r_0,1/s_0')\in\inte(\Oc_d)$ and $0<\vartheta<1$ with
\begin{align}\label{eq:torsion-interpolation-exponents-r38}
\left(\frac1r,\frac1{s'}\right)
=(1-\vartheta)\left(\frac12,\frac12\right)
+\vartheta\left(\frac1{r_0},\frac1{s_0'}\right).
\end{align}
Interpolation of the two uniform bounds in \eqref{eq:torsion-quantitative-improving-r38} gives the required translation modulus,
\begin{align}\label{eq:torsion-translation-modulus-r38}
\sup_{0\le\lambda\le1}
\|A_\lambda-\tau_yA_\lambda\|_{L^r\to L^{s'}}
\lesssim_{\gamma,r,s,\psi}|y|^{(1-\vartheta)/d},\qquad |y|\le1.
\end{align}
\par\Needspace{3\baselineskip}\smallskip\noindent\textit{Return to physical scales.} Set $A\coloneqq d(d+1)/2$ and $D_\lambda f\coloneqq f\circ\delta_\lambda$. Define
\begin{align}\label{eq:torsion-physical-annulus-r39}
A_{\gamma,\lambda}^{\psi} f(x)
&\coloneqq\int f(x-\gamma(\lambda u))\psi(u)\,\frac{du}{u},\\
A_{\gamma,\lambda}^{\psi}
&=D_\lambda^{-1}A_\lambda D_\lambda,
\qquad |\det\delta_\lambda|=\lambda^A.\notag
\end{align}
For $\vartheta$ as in \eqref{eq:torsion-interpolation-exponents-r38}, let $\eta\coloneqq(1-\vartheta)/d>0$. The norm change under $D_\lambda$ and \eqref{eq:torsion-translation-modulus-r38} give
\begin{align}\label{eq:torsion-physical-scaling-r39}
\|A_{\gamma,\lambda}^{\psi}\|_{L^r\to L^{s'}}
&\le C_{\gamma,r,s,\psi}\lambda^{-A(1/r-1/s')},\\
\|A_{\gamma,\lambda}^{\psi}-\tau_yA_{\gamma,\lambda}^{\psi}\|_{L^r\to L^{s'}}
&\le C_{\gamma,r,s,\psi}\lambda^{-A(1/r-1/s')}
\bigl|\delta_\lambda^{-1}y\bigr|^\eta,
\qquad \bigl|\delta_\lambda^{-1}y\bigr|\le1.\notag
\end{align}
Indeed, $D_\lambda\tau_y=\tau_{\delta_\lambda^{-1}y}D_\lambda$ and
\begin{align*}
\|D_\lambda\|_{L^r\to L^r}&=\lambda^{-A/r},\qquad
\|D_\lambda^{-1}\|_{L^{s'}\to L^{s'}}=\lambda^{A/s'}.
\end{align*}
\par\Needspace{3\baselineskip}\smallskip\noindent\textit{Sharp annular cutoffs.} For the pieces required by \cite[Section~3, Lemma~3.2]{CladekOu2018}, let $J$ be one interval contained in either $[-1,-1/2]$ or $[1/2,1]$, the zero extension of $u^{-1}\mathbf1_J(u)$ satisfies
\begin{align*}
\bigl\|u^{-1}\mathbf1_J\bigr\|_\infty&\le2,\qquad
\operatorname{Var}_{\R}\bigl(u^{-1}\mathbf1_J\bigr)\le4.
\end{align*}
Thus the Fourier estimate \eqref{eq:torsion-uniform-fourier-r36} and the improving and translation bounds in \eqref{eq:torsion-strong-improving-r40}--\eqref{eq:torsion-physical-scaling-r39} hold with constants independent of the endpoints of $J$ and $\lambda$. The two half-annuli are treated separately, which changes these constants by at most a fixed factor.

\par\Needspace{3\baselineskip}\smallskip\noindent\textit{Uniform truncations.} The stopping terms require uniform $L^p$ bounds. For $k\ge0$ let $\mu_k=\nu_k^+-\nu_k^-$ be the signed annular measure, with $\nu_k^\pm$ the positive half-annular measures. Then
\begin{align*}
\lambda=2^{-k},\qquad \nu_k^\pm(\R^d)=\log2,\qquad
\mu_k(\R^d)=0.
\end{align*}
In the Taylor frame define
\begin{align}\label{eq:torsion-annular-multiplier-r40}
U_k(\xi)&\coloneqq\max_{1\le j\le d}2^{-kj}|\xi_j|,\\
\|\mu_k\|_{\mathrm{TV}}+\|\nu_k^+\|_{\mathrm{TV}}+\|\nu_k^-\|_{\mathrm{TV}}&\le C_\gamma,\notag\\
|\widehat\mu_k(\xi)|&\le C_\gamma\min\{U_k(\xi),(1+U_k(\xi))^{-1/d}\},\notag\\
|\widehat\nu_k^\pm(\xi)|&\le C_\gamma(1+U_k(\xi))^{-1/d},\notag\\
2^{-d}U_k(\xi)&\le U_{k+1}(\xi)\le\tfrac12 U_k(\xi).\notag
\end{align}
The low-frequency estimate uses $\mu_k(\R^d)=0$, and the high-frequency estimate uses \eqref{eq:torsion-uniform-fourier-r36}. The same oscillatory argument controls $\nu_k^\pm$. Express these inequalities through the dual quasi-norm
\begin{align}\label{eq:torsion-dual-quasinorm-r40}
\rho_*(\xi)&\coloneqq\max_{1\le j\le d}|\xi_j|^{1/j},\qquad
\rho_*(\delta_{2^{-k}}^T\xi)=2^{-k}\rho_*(\xi),\\
|\widehat\mu_k(\xi)|&\lesssim_\gamma\min\{2^{-k}\rho_*(\xi),\,(2^{-k}\rho_*(\xi))^{-1/d}\}.\notag
\end{align}
The next lemma gives uniform truncation bounds even when the annular measures vary with scale.

\begin{lemma}[Uniform annular truncations]\label{lem:uniform-annular-truncations-r42}
Suppose that $\mu_k=\nu_k^+-\nu_k^-$, $k\ge0$, where $\nu_k^\pm$ are positive measures of uniformly bounded mass, supported in
\begin{align*}
\{x:|x_i|\le C2^{-ki},\ 1\le i\le d\}.
\end{align*}
Assume, with $\rho_*$ as in \eqref{eq:torsion-dual-quasinorm-r40}, that
\begin{align}\label{eq:annular-hypotheses-r42}
|\widehat\mu_k(\xi)|&\le C\min\{2^{-k}\rho_*(\xi),(2^{-k}\rho_*(\xi))^{-\epsilon}\},\\
|\widehat\nu_k^\pm(\xi)|&\le C(1+2^{-k}\rho_*(\xi))^{-\epsilon},\qquad \epsilon>0.\notag
\end{align}
Then, for every $1<p<\infty$,
\begin{align}\label{eq:annular-maximal-r42}
\left\|\sup_{k\ge0,\,\pm}\nu_k^\pm*|f|\right\|_{L^p}
+\left\|\sup_{k\ge0,\,\pm}\widetilde\nu_k^\pm*|f|\right\|_{L^p}
&\le C_p\|f\|_{L^p},\\
\sup_{0\le k_0\le k_1}\left\|\sum_{k=k_0}^{k_1}\mu_k*f\right\|_{L^p}
&\le C_p\|f\|_{L^p}.\notag
\end{align}
Here $\widetilde\nu(E)\coloneqq\nu(-E)$, and the constants depend only on the displayed quantitative bounds, $d,p,\epsilon$.
\end{lemma}

\begin{proof}
\noindent\textit{Cancellation and frequency localization.} Split each positive measure into a smooth part and a mean-zero remainder. Choose $\phi\ge0$ in $C_c^\infty$, $\int\phi=1$, and define
\begin{align*}
A&=\frac{d(d+1)}2,\qquad \phi_k(x)=2^{kA}\phi(\delta_{2^k}x),\\
m_k^\pm&=\nu_k^\pm(\R^d),\qquad
\sigma_k^\pm=\nu_k^\pm-m_k^\pm\phi_k\,dx,\qquad \sigma_k^\pm(\R^d)=0.
\end{align*}
By \eqref{eq:annular-hypotheses-r42},
\begin{align}\label{eq:annular-smooth-difference-r42}
|\widehat\sigma_k^\pm(\xi)|
\le C\min\{2^{-k}\rho_*(\xi),(2^{-k}\rho_*(\xi))^{-\epsilon_0}\},\qquad
\epsilon_0\coloneqq\min\{\epsilon,1\}>0.
\end{align}
Let $P_j$ be a smooth Littlewood--Paley decomposition for the dual dilations $\delta_{2^j}^T$, with Fourier support $\rho_*(\xi)\eqsim2^j$. For $j\in\Z$, define
\begin{align}\label{eq:annular-frequency-pieces-r42}
S_j^\pm f\coloneqq\left(\sum_{k\ge0}|\sigma_k^\pm*P_{k+j}f|^2\right)^{1/2}.
\end{align}
On $\operatorname{supp}\widehat P_{k+j}$, the multiplier in \eqref{eq:annular-smooth-difference-r42} is bounded by $C2^{-c|j|}$ for some $c>0$. Plancherel and the bounded frequency overlap give
\begin{align}\label{eq:annular-maximal-l2-r42}
\|S_j^\pm f\|_2&\le C2^{-c|j|}\|f\|_2,\\
\left\|\sup_{k\ge0}\nu_k^\pm*|f|\right\|_2
&\le C\|f\|_2,\notag
\end{align}
where the second inequality follows by summing in $j$ and bounding $\sup_k\phi_k*|f|$ by the anisotropic Hardy--Littlewood maximal operator. The same estimates hold for reflected measures.

\par\Needspace{3\baselineskip}\smallskip\noindent\textit{Maximal-function bootstrap.} Suppose the positive and reflected maximal operators are bounded on $L^q$ for $q>q_0$. To reach exponents below $2$ we require
\begin{align*}
\frac{p}{2-p}>q_0\quad\Longleftrightarrow\quad
\frac{2q_0}{q_0+1}<p<2.
\end{align*}
Jensen's inequality and duality then give
\begin{align}\label{eq:annular-vector-bootstrap-r42}
\left\|\left(\sum_k|\nu_k^\pm*g_k|^2\right)^{1/2}\right\|_p
&\le C_p\left\|\left(\sum_k|g_k|^2\right)^{1/2}\right\|_p.
\end{align}
Indeed, the dual estimate for $p'>2$ follows from
\begin{align*}
\sum_k|\widetilde\nu_k^\pm*h_k|^2
&\le\sum_k\widetilde\nu_k^\pm*|h_k|^2\,\nu_k^\pm(\R^d),\\
\int\sum_k\widetilde\nu_k^\pm*|h_k|^2\,w
&\le C\int\sum_k|h_k|^2\sup_k\nu_k^\pm*w,
\end{align*}
followed by H\"older's inequality with exponents $p'/2$ and $(p'/2)'=p/(2-p)$. The same estimate holds for $\phi_k$ by the anisotropic maximal theorem. Consequently, \eqref{eq:annular-frequency-pieces-r42}, the vector-valued estimate and the Littlewood--Paley square function imply
\begin{align*}
\|S_j^\pm f\|_p\le C_p\|f\|_p
\qquad\left(\frac{2q_0}{q_0+1}<p<2\right).
\end{align*}
Interpolate with \eqref{eq:annular-maximal-l2-r42} and sum the frequency shifts:
\begin{align*}
\|S_j^\pm f\|_p&\lesssim_p2^{-c_p|j|}\|f\|_p,\\
\Bigl\|\sup_k\nu_k^\pm*|f|\Bigr\|_p
&\lesssim_p\Bigl\|\sup_k\phi_k*|f|\Bigr\|_p
+\sum_{j\in\Z}\|S_j^\pm(|f|)\|_p\\
&\lesssim_p\|f\|_p+\sum_{j\in\Z}2^{-c_p|j|}\|f\|_p
\lesssim_p\|f\|_p.
\end{align*}
Starting from $q_0=2$, iterate
\begin{align}\label{eq:annular-bootstrap-sequence-r42}
q_{m+1}\coloneqq\frac{2q_m}{q_m+1},\qquad
q_m=\frac{2^{m+1}}{2^{m+1}-1}\downarrow1,\quad q_0=2.
\end{align}
This covers $1<p<2$. For $p\ge2$, interpolate the established $L^2$ and the uniform $L^\infty$ bounds.

\par\Needspace{3\baselineskip}\smallskip\noindent\textit{Signed measures and finite sums.} Apply the positive maximal bound to $\nu_k^\sigma$ and its reflection, $\sigma\in\{+,-\}$. Jensen's inequality and duality (at $q=p/(2-p)$ for $p<2$, at exponent $p'$ for $p>2$) yield
\begin{align}\label{eq:annular-positive-vector-r43}
\left\|\left(\sum_k|\nu_k^\sigma*g_k|^2\right)^{1/2}\right\|_p
+\left\|\left(\sum_k|\widetilde\nu_k^\sigma*g_k|^2\right)^{1/2}\right\|_p
\le C_p\left\|\left(\sum_k|g_k|^2\right)^{1/2}\right\|_p.
\end{align}
Since $\mu_k=\nu_k^+-\nu_k^-$, the triangle inequality in $\ell^2$ followed by the $L^p$ triangle inequality yields
\begin{align}\label{eq:annular-vector-all-p-r42}
\left\|\left(\sum_k|\mu_k*g_k|^2\right)^{1/2}\right\|_p
&\le\left\|\left(\sum_k|\nu_k^+*g_k|^2\right)^{1/2}\right\|_p
+\left\|\left(\sum_k|\nu_k^-*g_k|^2\right)^{1/2}\right\|_p\\
&\le C_p\left\|\left(\sum_k|g_k|^2\right)^{1/2}\right\|_p.\notag
\end{align}
For a fixed finite range $K=[k_0,k_1]\cap\Z$ let $T_{j,K}f\coloneqq\sum_{k\in K}\mu_k*P_{k+j}f$. By \eqref{eq:annular-vector-all-p-r42}, H\"older's inequality in $\ell^2$ and the dual Littlewood--Paley estimate,
\begin{align}\label{eq:annular-signed-lp-r42}
\|T_{j,K}f\|_p&\le C_p\|f\|_p,\\
\|T_{j,K}f\|_2&\le C2^{-c|j|}\|f\|_2,\notag
\end{align}
The second inequality follows by bounded frequency overlap. Interpolate at an auxiliary exponent on the opposite side of $p$ from $2$. For Schwartz $f$,
\begin{align*}
\|T_{j,K}f\|_p&\lesssim_p2^{-c_p|j|}\|f\|_p,\\
\sum_{k\in K}\mu_k*f&=\sum_{j\in\Z}T_{j,K}f,\\
\Bigl\|\sum_{k\in K}\mu_k*f\Bigr\|_p
&\le\sum_{j\in\Z}\|T_{j,K}f\|_p\lesssim_p\|f\|_p.
\end{align*}
The bound is uniform in $K$. Density completes the proof.
\end{proof}

Apply \Lemref{lem:uniform-annular-truncations-r42} using \eqref{eq:torsion-annular-multiplier-r40}--\eqref{eq:torsion-dual-quasinorm-r40} and the Taylor-frame support estimate. Here $\epsilon=1/d$. We obtain
\begin{align}\label{eq:torsion-uniform-lp-truncation-r40}
\sup_{0\le k_0\le k_1<\infty}
\left\|\sum_{k=k_0}^{k_1}\mu_k*f\right\|_{L^p}
&\le C_{p,\gamma}\|f\|_{L^p},\qquad1<p<\infty.
\end{align}
For $p=2$ we also have
\begin{align*}
\sup_\xi\sum_{k\ge0}|\widehat\mu_k(\xi)|\lesssim_\gamma1,
\qquad\sup_{k_0\le k_1}\Bigl\|\sum_{k=k_0}^{k_1}\mu_k*f\Bigr\|_2\lesssim_\gamma\|f\|_2.
\end{align*}
Each partial end-annulus has $L^p\to L^p$ norm at most $2\log2$ by Minkowski's inequality. Thus \eqref{eq:torsion-uniform-lp-truncation-r40} gives
\begin{align}\label{eq:torsion-truncation-lp-r40}
\sup_{0<\varepsilon<\rho\le1}
\left\|\int_{\varepsilon\le|t|\le\rho}
 f(\cdot-\gamma(t))\,\frac{dt}{t}\right\|_{L^p}
\le C_{p,\gamma}\|f\|_{L^p},\qquad 1<p<\infty.
\end{align}
\par\Needspace{3\baselineskip}\smallskip\noindent\textit{Stopping-time localization.} If $Q$ is a Taylor-frame adapted dyadic box of parameter length $2^q$ and $2^{q-N}\le1$, then choosing $N=N_{\mathrm{loc}}(\gamma)$ sufficiently large gives
\begin{align}\label{eq:torsion-support-stopping-r40}
|\langle\gamma(t),e_j\rangle|&\le C_\gamma|t|^j
\le\tfrac16\operatorname{side}_j(Q),\qquad |t|\le2^{q-N},\\
\operatorname{supp}\bigl(A_{\gamma,2^{q-N}}(f\mathbf1_{\frac13Q})\bigr)&\subseteq Q.\notag
\end{align}
Let $A_{\gamma,2^{q-N}}$ have annulus $2^{q-N-1}<|t|\le2^{q-N}$, where $2^{q-N}\le1$. The central-third partition in \Lemref{lem:corrected-grid} gives, for compactly supported $f$,
\begin{align}\label{eq:torsion-stopping-grid-reconstruction-r41}
f&=\sum_{\mathbf a\in\{0,1,2\}^d}
\sum_{Q\in\mathscr D_{\mathbf a}(q)}f\mathbf1_{\frac13Q}
\quad\text{a.e.},\\
A_{\gamma,2^{q-N}}f
&=\sum_{\mathbf a\in\{0,1,2\}^d}
\sum_{Q\in\mathscr D_{\mathbf a}(q)}
A_{\gamma,2^{q-N}}\bigl(f\mathbf1_{\frac13Q}\bigr).
\notag
\end{align}
By \eqref{eq:torsion-support-stopping-r40}, each summand is supported in its $Q$. The full $3^d$ shifts reconstruct $f$ in $L^p$. For $q_0\le q\le q_1$, the sum of annuli has endpoints $2^{q_0-N-1}$ and $2^{q_1-N}$, with at most two partial end-annuli. Equation \eqref{eq:torsion-truncation-lp-r40} yields
\begin{align}\label{eq:torsion-stopping-scale-interface-r41}
\left\|\sum_{q=q_0}^{q_1}A_{\gamma,2^{q-N}}f\right\|_{L^p}
\le C_{p,\gamma}\|f\|_{L^p},\qquad1<p<\infty.
\end{align}
We apply the stopping decomposition from \cite[Section~2, Lemma~2.3]{CladekOu2018}. Its four requirements are given by \eqref{eq:torsion-support-stopping-r40}, \eqref{eq:torsion-truncation-lp-r40}, \eqref{eq:torsion-physical-scaling-r39} and \Lemref{lem:corrected-grid}. After normalizing the parent averages to $1$, the good and bad parts on $Q_0$ satisfy
\begin{align*}
\supp g_i&\subseteq Q_0,\qquad \|g_i\|_\infty\lesssim1,\\
\|g_1\|_{s'}&\lesssim|Q_0|^{1/s'},\qquad
\|g_2\|_{r'}\lesssim|Q_0|^{1/r'},\\
\|b_1\|_r&\lesssim|Q_0|^{1/r},\qquad
\|b_2\|_s\lesssim|Q_0|^{1/s}.
\end{align*}
The $L^{s'}$ and $L^{r'}$ truncated bounds therefore control the mixed terms:
\begin{align}\label{eq:torsion-stopping-mixed-r40}
\|T_{\varepsilon,\rho}g_1\|_{s'}\|b_2\|_s
&\lesssim_{s,\gamma}|Q_0|,\\
\|b_1\|_r\|T_{\varepsilon,\rho}^*g_2\|_{r'}
&\lesssim_{r,\gamma}|Q_0|,\notag
\end{align}
Here $T_{\varepsilon,\rho}$ denotes symmetric truncation. The three stopping-time terms are controlled by
\begin{align*}
\text{good--good}&:\quad\eqref{eq:torsion-truncation-lp-r40}\ (p=2),\\
\text{mixed}&:\quad\eqref{eq:torsion-stopping-mixed-r40},\\
\text{bad--bad}&:\quad\eqref{eq:torsion-physical-scaling-r39}.
\end{align*}
All bounds are uniform in the annular scales. Each shift $\mathbf a$ produces an $\eta$-sparse family $\Sc_{\mathbf a}$. For a finite subcollection $\mathcal F\subseteq\bigcup_{\mathbf a}\Sc_{\mathbf a}$,
\begin{align}\label{eq:torsion-finite-sparse-merge-r42}
\sum_{Q\in\mathcal F}|Q|
&\le\eta^{-1}\sum_{\mathbf a}\left|\bigcup_{Q\in\mathcal F\cap\Sc_{\mathbf a}}Q\right|
\le\frac{3^d}{\eta}\left|\bigcup_{Q\in\mathcal F}Q\right|.
\end{align}
We use the measure-theoretic Hall theorem of Bollob\'as and Varopoulos \cite{BollobasVaropoulos1975}: if $(X,\mu)$ is atomless, $\{E_i\}_{i\in I}$ is a countable family of measurable sets of finite measure, and $a_i\ge0$ satisfy
\begin{align}\label{eq:nonatomic-hall-condition-r43}
\sum_{i\in I_0}a_i\le\mu\left(\bigcup_{i\in I_0}E_i\right)
\qquad\text{for every finite }I_0\subseteq I,
\end{align}
The theorem yields pairwise disjoint $F_i\subseteq E_i$, with $\mu(F_i)=a_i$. For $E_i=Q$ and $a_i=\eta3^{-d}|Q|$, \eqref{eq:torsion-finite-sparse-merge-r42} gives
\begin{align*}
\sum_{Q\in\mathcal F}a_Q&\le\biggl|\bigcup_{Q\in\mathcal F}Q\biggr|,\\
E_Q&\subseteq Q,\qquad |E_Q|=\eta3^{-d}|Q|,\qquad E_Q\cap E_{Q'}=\varnothing\ (Q\ne Q').
\end{align*}
Thus the countable union is $\eta3^{-d}$-sparse. Discarding repeated boxes costs at most $3^d$. The finite-scale stopping argument of \cite[Sections~2--3]{CladekOu2018} yields \eqref{eq:theoremB} after passing to the limit. Indeed, for $f\in C_c^\infty$,
\begin{align*}
|f(x-\gamma(t))-f(x-\gamma(-t))|&\lesssim_f |t|,\\
\left|\int_0^1\frac{f(x-\gamma(t))-f(x-\gamma(-t))}{t}\,dt\right|&<\infty.
\end{align*}
The uniform bound \eqref{eq:torsion-truncation-lp-r40} and density extend the principal-value operator to every $L^p$, $1<p<\infty$.

In the fixed Taylor frame, $\alpha_i=i$. \Thmref{thm:B} and \eqref{eq:adapted-continuity-r33} give
\begin{align*}
\|\Lambda\|_{L^p\times L^{p'}\to\C}\lesssim_{\gamma,p,r,s}1,\qquad r<p<s'.
\end{align*}
Now \Propref{prop:adapted-lifting} proves \eqref{eq:main-torsion} and the asymmetric dyadic bound.
\end{proof}

\medskip

\begin{proof}[Proof of \Corref{cor:matrix}]
The hypotheses $1\le r<p<s'$ are exactly the hypotheses needed in \Propref{prop:B}. Let $\overrightarrow H$ denote the corresponding operator $\Hvec$ or $\Hvecloc$. Let
\begin{align*}
E_m=\{x\in B(0,m):\|W(x)^{1/p}\|_{\op}
+\|W(x)^{-1/p}\|_{\op}\le m\}.
\end{align*}
As $m\to\infty$, $\mathbf1_{E_m}\to1$ almost everywhere. Truncating in space and magnitude gives the dense test classes
\begin{align*}
&\{\vecf\in L_c^\infty:W^{1/p}\vecf\in L_c^\infty\}\subseteq L^p(W),\\
&\{\vecG\in L_c^\infty:W^{-1/p}\vecG\in L_c^\infty\}
\subseteq L^{p'}(W^{-p'/p}).
\end{align*}
We prove the inequality on these classes and then pass to the closure. By \eqref{eq:duality}, normalize the dual function by
\begin{align}\label{eq:Gnorm}
        \|\vecG\|_{L^{p'}(W^{-p'/p})}=1
\end{align}
and prove
\begin{align}\label{eq:goal-matrix}
        |\langle \overrightarrow{H}\vecf,\vecG\rangle|
        \le C[W]_{p,r,s}^{\frac{s'-p}{s'(p-r)}+\frac1{p'}}\|\vecf\|_{L^p(W)}.
\end{align}
Define
\begin{align}\label{eq:F-G0}
        \vecF\coloneqq W^{1/p}\vecf,
        \qquad
        \vecGz\coloneqq W^{-1/p}\vecG.
\end{align}
Then
\begin{align}\label{eq:FGnorms}
        \|\vecF\|_{L^p}=\|\vecf\|_{L^p(W)},
        \qquad
        \|\vecGz\|_{L^{p'}}=\|\vecG\|_{L^{p'}(W^{-p'/p})}=1.
\end{align}
Apply the asymmetric dyadic estimate \eqref{eq:adapted-asymmetric-lifting} to $\vecf,\vecG$. It produces one sparse collection $\Sc_\star\subseteq\mathscr D_\star$ satisfying
\begin{align}\label{eq:use-cb}
|\langle\overrightarrow H\vecf,\vecG\rangle|
&\le C\sum_{Q\in\Sc_\star}|Q|\,
\llangle\vecf\rrangle_{r,Q^+}\circ\llangle\vecG\rrangle_{s,Q}\\
&=C\sum_{Q\in\Sc_\star}|Q|\,
\llangle W^{-1/p}\vecF\rrangle_{r,Q^+}\circ
\llangle W^{1/p}\vecGz\rrangle_{s,Q}.\notag
\end{align}
Using \Propref{prop:B} and \eqref{eq:FGnorms}, we have
\begin{align}\label{eq:embed-finish}
|\langle\overrightarrow H\vecf,\vecG\rangle|
&\le C[W]_{p,r,s}^{\frac{s'-p}{s'(p-r)}+\frac1{p'}}
\|\vecF\|_{L^p}\|\vecGz\|_{L^{p'}}\\
&=C[W]_{p,r,s}^{\frac{s'-p}{s'(p-r)}+\frac1{p'}}
\|\vecf\|_{L^p(W)}.\notag
\end{align}
This proves \eqref{eq:goal-matrix}. Taking the supremum over all $\vecG$ satisfying \eqref{eq:Gnorm} proves \eqref{eq:matrix-main} for the corresponding operator.
\end{proof}

\begin{proof}[Proof of \Corref{cor:sharpness}]
The sharpness proof has two parts: produce a divergent scalar pairing and bound all sparse forms on the same tests. For the monomial curve, choose $j$ with $\varepsilon_j\ne\varepsilon'_j$, let $\eta(t)=\gamma(t)$ for $t>0$, and take $L(z)=\varepsilon_jz_j$. Then
\begin{align*}
\det\bigl(\eta'(t),\ldots,\eta^{(d)}(t)\bigr)&\ne0\quad(t>0),\\
L(\eta(t))&=t^{\alpha_j}>0,\qquad
\|u_k\|_{L^r}=\|v_k\|_{L^s}=1.
\end{align*}
By \Lemref{lem:separated-knapp-r35}, the positive contribution satisfies
\begin{align}\label{eq:sharpness-monomial-tests-r34}
B_k\coloneqq\int_I\int_{\R^d}u_k(z-\gamma(t))v_k(z)\,dz\,\frac{dt}{t}\longrightarrow\infty.
\end{align}
The fixed support separation eliminates the opposite branch:
\begin{align*}
t<0&\Longrightarrow L(\gamma(t))<0,\\
L(F-E)>0&\Longrightarrow
\langle u_k(\cdot-\gamma(t)),v_k\rangle=0,\\
|\langle H_\gamma u_k,v_k\rangle|&\ge B_k\longrightarrow\infty.
\end{align*}

For a local curve, apply the same lemma to the limiting polynomial $\eta=\Gamma_0$ in \eqref{eq:torsion-moment-limit-r34}, with $L(z)=z_1$. Its nonzero diagonal coefficients give nonvanishing torsion, and $\eta_1(t)=\gamma_1'(0)t>0$ on $(1/2,1)$. We obtain nonnegative normalized functions $u_k,v_k$, supported in fixed bounded sets $E,F$ with $\inf_{z\in F-E}z_1>0$, and
\begin{align}\label{eq:sharpness-one-scale-moment-r34}
B_k\coloneqq\int_I\int_{\R^d}u_k(z-\Gamma_0(t))v_k(z)\,dz\,\frac{dt}{t}\longrightarrow\infty.
\end{align}
In both cases the supports are separated by a positive distance. The lemma proves the one-scale necessity directly from the three Knapp tests, without requiring the separated-support assertion from \cite[Section~4.2]{CladekOu2018} for a self-intersecting curve.

\par\Needspace{3\baselineskip}\smallskip\noindent\textit{Uniform sparse bound.} For either choice of $u_k,v_k$, write $\theta\coloneqq1/r+1/s-1>0$ and $A_\gamma\coloneqq\alpha_1+\cdots+\alpha_d$, where $\alpha_i=i$ in the local case. For every $\delta$-sparse collection $\Sc$ of genuine $\gamma$-cubes or their comparable dyadic approximations, only boxes meeting both $E$ and $F$ contribute. Define
\begin{align*}
\Sc_j(E,F)\coloneqq\{Q\in\Sc:2^j\le\ell(Q)<2^{j+1},\ Q\cap E\ne\varnothing,\ Q\cap F\ne\varnothing\}.
\end{align*}
Since $\operatorname{dist}(E,F)>0$, every $Q\in\Sc_j(E,F)$ has $\ell(Q)\ge c_{E,F,\gamma}>0$. We shall prove the uniform cardinality bound $\#\Sc_j(E,F)\lesssim_{E,F,\gamma,\delta}1$ and hence
\begin{align}\label{eq:sparse-sharpness-bound-r14}
\sum_{Q\in\Sc_j(E,F)}
|Q|\langle u_k\rangle_{r,Q}\langle v_k\rangle_{s,Q}
&\le\|u_k\|_{L^r}\|v_k\|_{L^s}
\sum_{Q\in\Sc_j(E,F)}|Q|^{-\theta}
\le C_{E,F,\gamma,\delta}\,2^{-jA_\gamma\theta}.
\end{align}
At each scale, the contributing boxes lie in an adapted box of volume $O_{E,F,\gamma}(2^{jA_\gamma})$. If $E_Q$ are the disjoint sparse witnesses, then
\begin{align*}
\delta\sum_{Q\in\Sc_j(E,F)}|Q|
\le\sum_{Q\in\Sc_j(E,F)}|E_Q|
\lesssim_{E,F,\gamma}2^{jA_\gamma}.
\end{align*}
Since $|Q|\eqsim_\gamma2^{jA_\gamma}$, at most $C_{E,\gamma}/\delta$ boxes contribute. This is also true for comparable dyadic approximations. Summing the geometric bound in \eqref{eq:sparse-sharpness-bound-r14} gives
\begin{align}\label{eq:sparse-sharpness-uniform-r14}
\sup_{k}\sup_{\Sc\text{ $\delta$-sparse}}\Lambda_{\Sc,r,s}(u_k,v_k)
\le C_{E,F,\gamma,r,s,\delta}.
\end{align}

For a monomial curve we take $f_k=u_k$ and $g_k=v_k$. For a local curve, define $A\coloneqq\sum_{i=1}^d i$ and, for $0<\lambda\le1$, set
\begin{align}\label{eq:sharpness-rescaled-tests-r34}
f_{k,\lambda}(x)&\coloneqq\lambda^{-A/r}u_k(\delta_\lambda^{-1}x),
&g_{k,\lambda}(x)&\coloneqq\lambda^{-A/s}v_k(\delta_\lambda^{-1}x).
\end{align}
The Jacobian of $\delta_\lambda$ is $\lambda^A$. Thus $\|f_{k,\lambda}\|_{L^r}=\|g_{k,\lambda}\|_{L^s}=1$. Since $\delta_\lambda$ maps genuine adapted cubes to genuine adapted cubes and preserves their sparsity parameter, the change of variables in the sparse form and \eqref{eq:sparse-sharpness-uniform-r14} give
\begin{align}\label{eq:sharpness-rescaled-sparse-r34}
\sup_{\Sc\text{ $\delta$-sparse}}\Lambda_{\Sc,r,s}(f_{k,\lambda},g_{k,\lambda})
\le C_{E,F,\gamma,r,s,\delta}\lambda^{-A\theta}.
\end{align}
For nonhomogeneous dyadic approximations, the same counting is applied directly to the scaled supports. The first-coordinate separation gives
\begin{align*}
Q\cap\delta_\lambda E\ne\varnothing,\quad
Q\cap\delta_\lambda F\ne\varnothing
&\Longrightarrow \ell(Q)\gtrsim\lambda,\\
2^j\gtrsim\lambda&\Longrightarrow
\#\Sc_j(\delta_\lambda E,\delta_\lambda F)
\lesssim_{\gamma,\delta}1.
\end{align*}
Consequently,
\begin{align}\label{eq:sharpness-dyadic-direct-r35}
\sup_{\Sc\text{ $\delta$-sparse dyadic}}\Lambda_{\Sc,r,s}(f_{k,\lambda},g_{k,\lambda})
\lesssim\sum_{2^j\gtrsim\lambda}2^{-jA\theta}\lesssim\lambda^{-A\theta}.
\end{align}

\par\Needspace{3\baselineskip}\smallskip\noindent\textit{The local pairing.} To estimate the local Hilbert form, let
\begin{align}\label{eq:sharpness-positive-convergence-r34}
B_k(\lambda)\coloneqq\int_I\int_{\R^d}u_k(z-\Gamma_\lambda(t))v_k(z)\,dz\,\frac{dt}{t}.
\end{align}
By the $C^d$ convergence in \eqref{eq:torsion-moment-limit-r34} and translation continuity in $L^1$, we have $B_k(\lambda)\to B_k$ as $\lambda\to0$ for each fixed $k$. The integrand over positive parameters is nonnegative. We next control the negative-parameter contribution. Since $\gamma_1'(0)>0$ in the Taylor frame, $\gamma_1(t)<0$ for $-\varepsilon_0<t<0$ and a fixed $\varepsilon_0>0$. As the first coordinate of $F-E$ is positive, the supports of $f_{k,\lambda}(\cdot-\gamma(t))$ and $g_{k,\lambda}$ are disjoint for these $t$. On $[-1,-\varepsilon_0]$, the torsion hypothesis implies $\gamma'(t)\ne0$. A finite cover by intervals on which some coordinate $\gamma_i$ is strictly monotone, with $|\gamma_i'|$ bounded below, gives
\begin{align}\label{eq:sharpness-far-negative-measure-r34}
\bigl|\{t\in[-1,-\varepsilon_0]:\gamma(t)\in\delta_\lambda(F-E)\}\bigr|
\le C_{E,F,\gamma}\lambda.
\end{align}
Indeed, the $i$th coordinate of $\delta_\lambda(F-E)$ has length at most $C_{E,F}\lambda^i\le C_{E,F}\lambda$, and a monotone coordinate with derivative bounded below controls the length of its preimage. It follows that the entire negative-parameter contribution, after extracting $\lambda^{-A\theta}$, has absolute value at most $C_k\lambda$, where $C_k$ may depend on $\|u_k\|_{L^1}\|v_k\|_{L^\infty}$ but not on $\lambda$.

Choose $\lambda_k\downarrow0$ sufficiently small that $B_k(\lambda_k)\ge B_k/2$ and $C_k\lambda_k\le B_k/4$. We then obtain
\begin{align}\label{eq:sharpness-local-pairing-r34}
\bigl|\langle H_{\gamma,\mathrm{loc}}f_{k,\lambda_k},g_{k,\lambda_k}\rangle\bigr|
&\ge\lambda_k^{-A\theta}\bigl(B_k(\lambda_k)-C_k\lambda_k\bigr)
\ge\frac14\lambda_k^{-A\theta}B_k.
\end{align}
The estimate allows the images of positive and negative parameter intervals to intersect away from the origin. Combining \eqref{eq:sharpness-rescaled-sparse-r34} with \eqref{eq:sharpness-local-pairing-r34} and $B_k\to\infty$ rules out a fixed finite sum of scalar sparse forms in the local case. In the monomial case the same contradiction follows from \eqref{eq:sharpness-monomial-tests-r34} and \eqref{eq:sparse-sharpness-uniform-r14}.

Finally, embed the scalar functions as $\vecf=f e_1$ and $\vecg=g e_1$. For every adapted box $Q$,
\begin{align}\label{eq:one-coordinate-cb}
\llangle f e_1\rrangle_{r,Q}\circ\llangle g e_1\rrangle_{s,Q}
=\langle f\rangle_{r,Q}\langle g\rangle_{s,Q}.
\end{align}
Thus a uniform convex body estimate involving any fixed number $J$ of sparse families would imply the corresponding scalar estimate, contrary to the preceding constructions. This proves the corollary for both classes of curves.
\end{proof}

\section{Proofs of the auxiliary results}\label{sec:auxiliary-estimates}

We prove the grid lemma, the Knapp tests, the convex body lifting, the endpoint obstruction and the matrix embedding, in that order.
\begin{proof}[Proof of \Lemref{lem:corrected-grid}]
We first verify nestedness. For $a\in\{0,1,2\}$, the intervals $I^a_{\ell,m}$ partition $\R$, and their endpoints satisfy
\begin{align}\label{eq:alternating-nesting-proof}
2^{\ell+1}\bigl(m+(-1)^{\ell+1}a/3\bigr)
=2^\ell\bigl(2m-2(-1)^\ell a/3\bigr)
=2^\ell\bigl((2m-(-1)^\ell a)+(-1)^\ell a/3\bigr).
\end{align}
The endpoint identity gives nested intervals, hence nested Cartesian products. The remaining scale properties are
\begin{align*}
\nu_j(k)&\nearrow\infty\quad(k\to\infty),\qquad
\nu_j(k)\searrow-\infty\quad(k\to-\infty),\\
\nu_1(k+1)-\nu_1(k)&=1,
\end{align*}
and the parent-volume ratio is
\begin{align}\label{eq:adapted-parent-ratio}
\frac{|Q^{(1)}|}{|Q|}
=2^{\sum_{j=1}^d(\nu_j(k+1)-\nu_j(k))}
\le 2^{\sum_{j=1}^d\lceil\alpha_j\rceil},
\qquad Q\in\mathscr D_{\mathbf a}(k).
\end{align}
Here $Q^{(1)}$ is the unique box of generation $k+1$ containing $Q$.

For a fixed $\ell$, the central third of $I^a_{\ell,m}$ is
\begin{align}\label{eq:central-third-identity}
\tfrac13 I^a_{\ell,m}
=2^\ell\bigl[m+(-1)^\ell a/3+1/3,\,
 m+(-1)^\ell a/3+2/3\bigr).
\end{align}
As $a$ runs through $0,1,2$, these are precisely the consecutive intervals of length $2^\ell/3$, up to relabeling $m$. Their Cartesian products therefore partition $\R^d$ when $\mathbf a$ runs through $\{0,1,2\}^d$ at any fixed $k$. Finally,
\begin{align}\label{eq:adapted-scale-sandwich}
2^{-1}2^{k\alpha_j}<2^{\nu_j(k)}\le2^{k\alpha_j}
\end{align}
shows that every box has bounded eccentricity relative to the $\gamma$-cube of scale $2^k$ with the same center. This proves all the assertions.
\end{proof}

For $\mathscr D_\star=\mathscr D_{(1,\ldots,1)}$, the boxes containing the origin expand in every coordinate as $k\to\infty$. More precisely, for each such box $P_k$,
\begin{align*}
\operatorname{dist}(0,\partial P_k)\ge\tfrac13\min_{1\le j\le d}\operatorname{side}_j(P_k)
\longrightarrow\infty.
\end{align*}
Hence every compact support lies in one root $P_k$, as required in \Propref{prop:adapted-lifting}.

\subsection{Proof of \Lemref{lem:separated-knapp-r35}: three one-scale tests}\label{subsec:knapp-proof-r35}
\begin{proof}[Proof of \Lemref{lem:separated-knapp-r35}]
Write $x\coloneqq1/r$, $y\coloneqq1/s'$ and $A\coloneqq d(d+1)/2$. For $0<x,y<1$ and $x>y$, membership in $\overline{\Oc_d}$ is equivalent to
\begin{align}\label{eq:knapp-three-edges-r35}
y\ge\frac{d-1}{d}x,\qquad
y\ge\frac{dx-1}{d-1},\qquad
x-y\le\frac1A.
\end{align}
Indeed, the first and third boundary lines meet at the lower vertex $(2/(d+1),2(d-1)/(d(d+1)))$, and the second and third meet at $((d^2-d+2)/(d(d+1)),(d-1)/(d+1))$. Together with $(0,0)$ and $(1,1)$, they delimit exactly the polygon in \eqref{eq:trapezoid}. Thus at least one inequality in \eqref{eq:knapp-three-edges-r35} is violated.

Choose an interval $I\Subset(1/2,1)$ containing $t_0$ so short that $L(\eta(t))\ge 3c>0$ for $t\in I$. Write $a\coloneqq\eta(t_0)$. All the sets below lie in fixed small neighborhoods of $0$ and $a$, and we can choose their sizes so that $L(F-E)\ge c$ and $\operatorname{dist}(E,F)>0$. We write $N_\delta(S)$ for the Euclidean $\delta$-neighborhood of a set $S$. A bounded smooth curve segment has
\begin{align}\label{eq:knapp-tube-volume-r35}
|N_{2\delta}(\eta(I))|+|N_{2\delta}(a-\eta(I))|\lesssim_{\eta,I}\delta^{d-1},\qquad 0<\delta<\delta_0.
\end{align}
This follows by covering the parameter interval by $O(\delta^{-1})$ subintervals whose images have diameter $O(\delta)$.

\par\Needspace{3\baselineskip}\smallskip\noindent\textit{First boundary.} If $dy<(d-1)x$, take $u_\delta\coloneqq\mathbf1_{N_{2\delta}(a-\eta(I))}$ and $v_\delta\coloneqq\mathbf1_{B(a,\delta)}$. For $z\in B(a,\delta)$ and $t\in I$, we have $z-\eta(t)\in N_{2\delta}(a-\eta(I))$. Since $1/t\eqsim1$ on $I$, we obtain
\begin{align}\label{eq:knapp-tube-first-r35}
\frac{\displaystyle\int_I\!\int u_\delta(z-\eta(t))v_\delta(z)\,dz\,\frac{dt}{t}}{\|u_\delta\|_{L^r}\|v_\delta\|_{L^s}}
\gtrsim\frac{\delta^d}{\delta^{(d-1)/r}\delta^{d/s}}
=\delta^{dy-(d-1)x}\longrightarrow\infty.
\end{align}

\par\Needspace{3\baselineskip}\smallskip\noindent\textit{Second boundary.} If $1-dx+(d-1)y<0$, take $u_\delta\coloneqq\mathbf1_{B(0,\delta)}$ and $v_\delta\coloneqq\mathbf1_{N_{2\delta}(\eta(I))}$. For $z\in B(0,\delta)$ and $t\in I$, the point $z+\eta(t)$ lies in the latter neighborhood. Fubini's theorem then gives
\begin{align}\label{eq:knapp-tube-dual-r35}
\frac{\displaystyle\int_I\!\int u_\delta(z-\eta(t))v_\delta(z)\,dz\,\frac{dt}{t}}{\|u_\delta\|_{L^r}\|v_\delta\|_{L^s}}
\gtrsim\frac{\delta^d}{\delta^{d/r}\delta^{(d-1)/s}}
=\delta^{1-dx+(d-1)y}\longrightarrow\infty.
\end{align}

\par\Needspace{3\baselineskip}\smallskip\noindent\textit{Improving boundary.} If $x-y>1/A$, the nonzero determinant $\det(\eta'(t_0),\ldots,\eta^{(d)}(t_0))$ shows that the vectors $v_j\coloneqq\eta^{(j)}(t_0)/j!$, $1\le j\le d$, form a basis. In this basis let
\begin{align}\label{eq:knapp-anisotropic-test-r35}
D_\delta\coloneqq\left\{\sum_{j=1}^d v_j z_j:|z_j|<\delta^j\right\},\quad
u_\delta\coloneqq\mathbf1_{3D_\delta},\quad
v_\delta\coloneqq\mathbf1_{a+D_\delta},\quad
|D_\delta|\eqsim_{\eta}\delta^A.
\end{align}
Taylor's formula in $C^d$ implies $\eta(t_0+h)-a=\sum_{j=1}^dv_jh^j+o(|h|^d)$. For $|h|\le c_0\delta$ and sufficiently small fixed $c_0>0$, this difference belongs to $D_\delta$. Thus $z-\eta(t_0+h)\in3D_\delta$ whenever $z\in a+D_\delta$, and
\begin{align}\label{eq:knapp-anisotropic-lower-r35}
\frac{\displaystyle\int_I\!\int u_\delta(z-\eta(t))v_\delta(z)\,dz\,\frac{dt}{t}}{\|u_\delta\|_{L^r}\|v_\delta\|_{L^s}}
\gtrsim\frac{\delta|D_\delta|}{|D_\delta|^{1/r+1/s}}
\eqsim\delta^{1-A(x-y)}\longrightarrow\infty.
\end{align}

Choose $\delta_k\downarrow0$ in the applicable case and normalize the two indicator functions by their $L^r$ and $L^s$ norms. Their fixed support neighborhoods can be chosen simultaneously after fixing $I$, and each construction has the required separation. This proves \eqref{eq:knapp-unnormalized-r35}.
\end{proof}

\subsection{Proof of \Propref{prop:adapted-lifting}: asymmetric randomization}\label{subsec:adapted-lifting-proof}
\begin{proof}[Proof of \Propref{prop:adapted-lifting}]
We prove the dyadic lifting in five steps. Set $A_\gamma^0\coloneqq\sum_{i=1}^d\alpha_i$. First, the dyadic and enlarged volumes satisfy
\begin{align}\label{eq:adapted-lift-volume}
|Q|\eqsim_\gamma 2^{k A_\gamma^0},\qquad
|Q^+|=A_\gamma^{A_\gamma^0}2^{k A_\gamma^0}
\eqsim_\gamma|Q|,\qquad Q\in\mathscr D_\star(k).
\end{align}

We first fix the enlargement constant in a way compatible with every generation. For $Q\in\mathscr D_\star(k)$ write $L_i(k)\coloneqq 2^{\nu_i(k)}$ for its $i$th side-length. Since all $\alpha_i\ge1$ after the normalization, fix $A_\gamma\ge8$. Then, for $1\le i\le d$,
\begin{align}\label{eq:adapted-dilation-conditions}
A_\gamma^{\alpha_i}\ge 1+2^{\alpha_i+1},\qquad
A_\gamma^{\alpha_i}(1-2^{-\alpha_i})\ge1.
\end{align}
If $Q\subseteq P$ are boxes of generations $k\le h$, then the nesting of the coordinate intervals gives
\begin{align}\label{eq:adapted-center-nesting}
2|c_i(Q)-c_i(P)|\le L_i(h)-L_i(k),\qquad 1\le i\le d.
\end{align}
For $k<h$, the scale difference and \eqref{eq:adapted-dilation-conditions} imply
\begin{align}\label{eq:adapted-plus-nesting}
L_i(h)-L_i(k)
&\le 2^{h\alpha_i}
\le A_\gamma^{\alpha_i}(2^{h\alpha_i}-2^{k\alpha_i}),\\
Q^+&\subseteq P^+.\notag
\end{align}
The inclusion is immediate when $k=h$, since then $Q=P$.

\par\Needspace{3\baselineskip}\smallskip\noindent\textbf{Step 1. Reduction of the scalar sparse form to one grid.}
For each genuine adapted box $R$ of scale $\ell$, choose a generation $k$ with $2^k\le\ell<2^{k+1}$. Let $\mathcal F(R)$ be the boxes of generation $k$ that meet $R$. This family covers $R$ and has at most $M_\gamma$ members, by \eqref{eq:adapted-scale-sandwich}. For every $Q\in\mathcal F(R)$, the overlap of the coordinate intervals implies
\begin{align}\label{eq:adapted-genuine-cover}
2|c_i(R)-c_i(Q)|\le\ell^{\alpha_i}+L_i(k),\qquad
2|c_i(R)-c_i(Q)|+\ell^{\alpha_i}
\le (2^{\alpha_i+1}+1)2^{k\alpha_i}
\le A_\gamma^{\alpha_i}2^{k\alpha_i}.
\end{align}
Thus $R\subseteq Q^+$ for every $Q\in\mathcal F(R)$. Choose $Q_R\in\mathcal F(R)$ maximizing $\langle g\rangle_{s,Q}$. Since $|Q_R|\eqsim_\gamma|R|$ and $R$ is covered by the bounded family $\mathcal F(R)$, we obtain
\begin{align}\label{eq:adapted-scalar-reduction}
|R|\langle f\rangle_{r,R}\langle g\rangle_{s,R}
&\le C_{\gamma,r,s}|R|
\langle f\rangle_{r,Q_R^+}\langle g\rangle_{s,Q_R}.
\end{align}
Let $\mathcal S_0$ be an $\eta$-sparse family of genuine boxes for the scalar form. For its witnesses $E_R\subseteq R$, we have $|E_R|\ge\eta|R|$. Define $\mathcal T=\{Q_R:R\in\mathcal S_0\}\subseteq\mathscr D_\star$. If $Q_R=P$, then $R\subseteq P^+$ and
\begin{align}\label{eq:adapted-multiplicity}
\sum_{R\in\mathcal S_0:Q_R=P}|R|
&\le\eta^{-1}\sum_{R:Q_R=P}|E_R|
\le\eta^{-1}|P^+|
\lesssim_\gamma\eta^{-1}|P|.
\end{align}
Sum all scalar coefficients assigned to the same $P$ before using sparsity. Equations \eqref{eq:adapted-scalar-reduction} and \eqref{eq:adapted-multiplicity} give
\begin{align}\label{eq:adapted-multiplicity-form}
\sum_{\substack{R\in\mathcal S_0\\Q_R=P}}
|R|\langle f\rangle_{r,R}\langle g\rangle_{s,R}
\le C_{\gamma,r,s}\eta^{-1}|P|
\langle f\rangle_{r,P^+}\langle g\rangle_{s,P}.
\end{align}
For any dyadic $P$, if $Q_R\subseteq P$, then $R\subseteq Q_R^+\subseteq P^+$ by \eqref{eq:adapted-plus-nesting}. For each $Q\in\mathcal T$ select one genuine box $R(Q)\in\mathcal S_0$ with $Q_{R(Q)}=Q$. These representatives are distinct, have $|Q|\eqsim_\gamma|R(Q)|$, and possess pairwise disjoint witnesses inside $P^+$ whenever $Q\subseteq P$. Consequently,
\begin{align}\label{eq:adapted-carleson-transfer}
\sum_{Q\in\mathcal T:Q\subseteq P}|Q|
&\lesssim_\gamma\sum_{Q\in\mathcal T:Q\subseteq P}|R(Q)|
\le\eta^{-1}|P^+|
\lesssim_\gamma\eta^{-1}|P|.
\end{align}
Here $|Q_R|\eqsim_\gamma|R|$, and the representatives used in the Carleson estimate have their witnesses inside $P^+$. On a nonatomic measure space the Carleson condition in \eqref{eq:adapted-carleson-transfer} gives an $\eta_1$-sparse family $\mathcal T$ with $\eta_1>0$ depending only on $\gamma,\eta$. Thus every scalar pair obeys the asymmetric dyadic sparse bound
\begin{align}\label{eq:adapted-scalar-asymmetric}
|\Lambda(f,g)|\le C_{\gamma,r,s,\eta}
\sum_{Q\in\mathcal T}|Q|\langle f\rangle_{r,Q^+}
\langle g\rangle_{s,Q}.
\end{align}

\par\Needspace{3\baselineskip}\smallskip\noindent\textbf{Step 2. Uniform localization at a top box.}
Fix $P\in\mathscr D_\star$ and suppose that $f$ is supported in $P^+$ while $g$ is supported in $P$. The dyadic boxes in \eqref{eq:adapted-scalar-asymmetric} with a nonzero $g$ average either lie inside $P$ or contain it. For $\theta\coloneqq1/r+1/s-1>0$, the ancestors form a geometric tail,
\begin{align}\label{eq:adapted-ancestor-tail}
\sum_{Q\supsetneq P}|Q|\langle f\rangle_{r,Q^+}\langle g\rangle_{s,Q}
&\le C_{\gamma,r,s}\|f\|_{L^r(P^+)}\|g\|_{L^s(P)}
\sum_{Q\supsetneq P}|Q|^{-\theta}\\
&\lesssim_{\gamma,r,s}|P|\langle f\rangle_{r,P^+}
\langle g\rangle_{s,P}.\notag
\end{align}
Since $\alpha_1=1$, each strict ancestor has at least twice the preceding volume. Hence
\begin{align*}
|P^{(h)}|\ge2^h|P|,\qquad
\sum_{h\ge1}|P^{(h)}|^{-\theta}
\le\frac{2^{-\theta}}{1-2^{-\theta}}|P|^{-\theta}.
\end{align*}
Adding $P$ to the descendant Carleson family increases its Carleson constant by at most one. Thus \eqref{eq:adapted-ancestor-tail} is absorbed into the top-box term uniformly in $P$.

\par\Needspace{3\baselineskip}\smallskip\noindent\textbf{Step 3. A stopping family with controlled enlarged overlap.}
We first use the Auerbach reduction \cite[Lemma~2.1]{LaukkarinenLorist2026} on the convex averages over $P^+$ and $P$. We obtain an invertible complex matrix $A_P$ and transformed vector fields with components $f_j^P,g_j^P$ defined by
\begin{align}\label{eq:adapted-lift-au-basis}
\vecf^P\coloneqq A_P\vecf,\qquad \vecg^P\coloneqq(A_P^{-1})^{*}\vecg.
\end{align}
The coordinatewise pairing is unchanged:
\begin{align}\label{eq:adapted-lift-sesquilinear-invariance}
\sum_{j=1}^n\Lambda(f_j^P,g_j^P)=\sum_{j=1}^n\Lambda(f_j,g_j).
\end{align}
Apply the complex Auerbach reduction of \cite[Lemma~2.1]{LaukkarinenLorist2026} to $\widetilde\Lambda(f,h)\coloneqq\Lambda(f,\overline h)$, then conjugate $h$. The normalization is \eqref{eq:adapted-lift-auerbach}. The weak-type bounds for $M_r^\gamma,M_s^\gamma$ allow a threshold $C_{\gamma,n}$ with the stopping set
\begin{align}\label{eq:adapted-lift-stop-set}
E_P\coloneqq\bigcup_{j=1}^n
\left\{M_r^\gamma(f_j^P\mathbf1_{P^+})>
C_{\gamma,n}\langle f_j^P\rangle_{r,P^+}\right\}
\cup\bigcup_{j=1}^n
\left\{M_s^\gamma(g_j^P\mathbf1_P)>
C_{\gamma,n}\langle g_j^P\rangle_{s,P}\right\}.
\end{align}
Set $B_\gamma\coloneqq5A_\gamma$ and let $R^{++}$ have center $c_R$ and side-lengths $(B_\gamma2^{k(R)})^{\alpha_i}$. Choose the maximal dyadic descendants $R\subseteq P$ with $R^{++}\subseteq E_P$, denoted by $\mathcal F_P$. Thus
\begin{align*}
R\in\mathcal F_P&\Longrightarrow R\subseteq R^{++}\subseteq E_P,\\
R,R'\in\mathcal F_P,\ R\ne R'&\Longrightarrow R\cap R'=\varnothing.
\end{align*}
The desired packing estimate is
\begin{align}\label{eq:adapted-lift-packing}
\sum_{R\in\mathcal F_P}|R|\le|E_P|\le\tfrac12|P|.
\end{align}
We verify the enlarged overlap. Suppose $R,R'\in\mathcal F_P$ have generations $k<h$ and $R^+\cap(R')^+\ne\varnothing$. If $R^{(1)}$ denotes the immediate dyadic parent of $R$, then the coordinate-center estimates give
\begin{align}\label{eq:adapted-overlap-centers}
2|c_i(R^{(1)})-c_i(R')|
&\le A_\gamma^{\alpha_i}(2^{k\alpha_i}+2^{h\alpha_i})
+2^{(k+1)\alpha_i},\qquad 1\le i\le d.
\end{align}
We may take $L_\gamma\coloneqq2$ in the comparison of generations. Indeed, since $\alpha_i\ge1$ and $B_\gamma\ge40$,
\begin{align}\label{eq:adapted-overlap-gap}
\frac{A_\gamma^{\alpha_i}(1+2^{-2\alpha_i})
+(1+B_\gamma^{\alpha_i})2^{-\alpha_i}}{B_\gamma^{\alpha_i}}
&=5^{-\alpha_i}(1+2^{-2\alpha_i})
+2^{-\alpha_i}+(2B_\gamma)^{-\alpha_i}\\
&\le\frac14+\frac12+\frac1{80}<1.
\end{align}
If $h-k\ge L_\gamma$, this inequality shows, coordinate by coordinate, that
\begin{align}\label{eq:adapted-overlap-parent}
(R^{(1)})^{++}\subseteq (R')^{++}\subseteq E_P.
\end{align}
Since $R^{(1)}\subseteq P$, \eqref{eq:adapted-overlap-parent} contradicts the maximality of $R$. Hence, whenever $x\in R^+$ for some $R\in\mathcal F_P$, all such $R$ have generations differing by at most one:
\begin{align*}
\max\{k(R):R\in\mathcal F_P,\ x\in R^+\}
-\min\{k(R):R\in\mathcal F_P,\ x\in R^+\}\le1.
\end{align*}
At each generation, the coordinate lattice spacing is $L_i(k)>2^{k\alpha_i-1}$. Thus
\begin{align*}
\#\{R\in\mathcal F_P\cap\mathscr D_\star(k):x\in R^+\}
\le\prod_{i=1}^d(3+2A_\gamma^{\alpha_i}),
\end{align*}
and the two-generation bound yields
\begin{align}\label{eq:adapted-lift-bounded-overlap}
\sum_{R\in\mathcal F_P}\mathbf1_{R^+}
\le L_\gamma\prod_{i=1}^d(3+2A_\gamma^{\alpha_i})
\coloneqq C_\gamma.
\end{align}
We next control all dyadic $Q\subseteq P$ that remain outside the stopping boxes. For these $Q$ the required inequalities are
\begin{align}\label{eq:adapted-lift-ancestor-averages}
\langle f_j^P\rangle_{r,Q^+}\lesssim_{\gamma,n}
\langle f_j^P\rangle_{r,P^+},\qquad
\langle g_j^P\rangle_{s,Q}\lesssim_{\gamma,n}
\langle g_j^P\rangle_{s,P}.
\end{align}
The stopping definition provides a point $z$ such that
\begin{align*}
Q\not\subseteq\bigcup_{R\in\mathcal F_P}R
&\Longrightarrow Q^{++}\not\subseteq E_P
\Longrightarrow z\in Q^{++}\setminus E_P,\\
Q\cup Q^+&\subseteq B_\gamma(z,C_\gamma2^{k(Q)}),\qquad
|B_\gamma(z,C_\gamma2^{k(Q)})|\lesssim_\gamma|Q|.
\end{align*}
Consequently, for each $j$,
\begin{align*}
\langle f_j^P\rangle_{r,Q^+}
&\lesssim_\gamma M_r^\gamma(f_j^P\mathbf1_{P^+})(z)
\le C_{\gamma,n}\langle f_j^P\rangle_{r,P^+},\\
\langle g_j^P\rangle_{s,Q}
&\lesssim_\gamma M_s^\gamma(g_j^P\mathbf1_P)(z)
\le C_{\gamma,n}\langle g_j^P\rangle_{s,P}.
\end{align*}

\par\Needspace{3\baselineskip}\smallskip\noindent\textbf{Step 4. Randomization and convex bodies.}
Let $\varepsilon_R\in\{0,1\}$ be independent Bernoulli variables of equal probabilities for $R\in\mathcal F_P$. Set $G_P\coloneqq P\setminus\bigcup_{R\in\mathcal F_P}R$ and
\begin{align}\label{eq:adapted-lift-random-weights}
m_\varepsilon(x)&\coloneqq\mathbf1_{P^+}(x)
\prod_{R\in\mathcal F_P:x\in R^+}2\varepsilon_R,\\
n_\varepsilon(x)&\coloneqq\mathbf1_{G_P}(x)
+2\sum_{R\in\mathcal F_P}(1-\varepsilon_R)\mathbf1_R(x).
\end{align}
The bounded-overlap estimate gives
\begin{align*}
0\le m_\varepsilon\le2^{C_\gamma},\qquad0\le n_\varepsilon\le2.
\end{align*}
Independence and sesquilinearity yield the identity used in the iteration:
\begin{align}\label{eq:adapted-lift-random-identity}
\mathbb E\Lambda(m_\varepsilon f,n_\varepsilon g)
=\Lambda(f\mathbf1_{P^+},g\mathbf1_P)
-\sum_{R\in\mathcal F_P}
\Lambda(f\mathbf1_{R^+},g\mathbf1_R).
\end{align}
The purpose of randomization is to annihilate every descendant of a stopping box. By \eqref{eq:adapted-plus-nesting},
\begin{align*}
Q\subseteq R\in\mathcal F_P&\Longrightarrow Q^+\subseteq R^+,\\
\varepsilon_R=0&\Longrightarrow m_\varepsilon\mathbf1_{Q^+}=0,\\
\varepsilon_R=1&\Longrightarrow n_\varepsilon\mathbf1_Q=0,\\
\langle m_\varepsilon f_j^P\rangle_{r,Q^+}
\langle n_\varepsilon g_j^P\rangle_{s,Q}&=0.
\end{align*}
Fix $\varepsilon$ and $j$, and let $\mathcal T_{\varepsilon,j}$ be the scalar sparse family supplied by \eqref{eq:adapted-scalar-asymmetric}. Step~2 absorbs its ancestors into the term for $P$. For the remaining boxes, \eqref{eq:adapted-lift-ancestor-averages} and \eqref{eq:adapted-carleson-transfer} yield
\begin{align}\label{eq:adapted-random-local-carleson-r36}
&\sum_{\substack{Q\in\mathcal T_{\varepsilon,j},\,Q\subseteq P\\Q\not\subseteq R\ \text{for every }R\in\mathcal F_P}}
|Q|\langle m_\varepsilon f_j^P\rangle_{r,Q^+}
\langle n_\varepsilon g_j^P\rangle_{s,Q}\notag\\
&\quad\le C_{\gamma,n}\langle f_j^P\rangle_{r,P^+}\langle g_j^P\rangle_{s,P}
\sum_{\substack{Q\in\mathcal T_{\varepsilon,j}\\Q\subseteq P}}|Q|
\le C_{\gamma,n,r,s,\eta}|P|\langle f_j^P\rangle_{r,P^+}\langle g_j^P\rangle_{s,P}.
\end{align}
The bound holds for each $\varepsilon$ and every selected $\mathcal T_{\varepsilon,j}$. For a finite stopping family the Bernoulli identities are
\begin{align}\label{eq:adapted-random-cancellation-r36}
\mathbb E[m_\varepsilon(x)]&=\mathbf1_{P^+}(x),\\
2\mathbb E[(1-\varepsilon_R)m_\varepsilon(x)]
&=\mathbf1_{P^+\setminus R^+}(x),\qquad R\in\mathcal F_P.
\end{align}
For finite stopping families, \eqref{eq:adapted-random-cancellation-r36} gives \eqref{eq:adapted-lift-random-identity}. The passage to countable families follows from $0\le m_\varepsilon\le2^{C_\gamma}$, $0\le n_\varepsilon\le2$ and
\begin{align}\label{eq:adapted-random-tail-absolute}
\sum_{R\in\mathcal F_P}
|\Lambda(f\mathbf1_{R^+},g\mathbf1_R)|
&\le C_{\gamma,p}\|\Lambda\|\|f\|_{\infty}\|g\|_{\infty}
\sum_{R\in\mathcal F_P}|R|<\infty.
\end{align}
The right side is finite by \eqref{eq:adapted-lift-packing}. Dominated convergence in $L^p\times L^{p'}$ extends \eqref{eq:adapted-lift-random-identity} to the full stopping collection.

For the Auerbach coordinates \eqref{eq:adapted-lift-au-basis}, the normalization gives
\begin{align}\label{eq:adapted-lift-auerbach}
\sum_{j=1}^n\langle f_j^P\rangle_{r,P^+}
\langle g_j^P\rangle_{s,P}
\le n\llangle\vecf\rrangle_{r,P^+}\circ
\llangle\vecg\rrangle_{s,P}.
\end{align}
For degenerate convex averages, use an Auerbach basis on the complex span and extend it to $\C^n$. Sum \eqref{eq:adapted-random-local-carleson-r36} over $j$ and apply \eqref{eq:adapted-lift-auerbach}. Then \eqref{eq:adapted-lift-random-identity} gives
\begin{align}\label{eq:adapted-lift-local-vector}
\Bigl|\Lambda(\vecf\mathbf1_{P^+},\vecg\mathbf1_P)
-\sum_{R\in\mathcal F_P}\Lambda(\vecf\mathbf1_{R^+},\vecg\mathbf1_R)\Bigr|
\le C_{\gamma,r,s,n}|P|\llangle\vecf\rrangle_{r,P^+}
\circ\llangle\vecg\rrangle_{s,P}.
\end{align}
Iterate \eqref{eq:adapted-lift-local-vector} from a root $P_0$. With $\mathcal G_h$ denoting the stopping boxes at depth $h$, \eqref{eq:adapted-lift-packing} implies
\begin{align*}
\mathcal G_0&=\{P_0\},\qquad
\sum_{Q\in\mathcal G_h}|Q|\le2^{-h}|P_0|,\qquad h\ge0.
\end{align*}
The last remainder is controlled directly by the $L^p\times L^{p'}$ norm $\|\Lambda\|$. No overlap assumption on boxes belonging to distinct parents is used:
\begin{align}\label{eq:adapted-generation-remainder-r33}
\sum_{Q\in\mathcal G_h}
\bigl|\Lambda(\vecf\mathbf1_{Q^+},\vecg\mathbf1_Q)\bigr|
&\le\|\Lambda\|\sum_{Q\in\mathcal G_h}\sum_{j=1}^n
\|f_j\mathbf1_{Q^+}\|_{L^p}\|g_j\mathbf1_Q\|_{L^{p'}}\\
&\le C_{\gamma,n,p}\|\Lambda\|\|\vecf\|_\infty\|\vecg\|_\infty
\sum_{Q\in\mathcal G_h}|Q|\le C_{\gamma,n,p}\|\Lambda\|\|\vecf\|_\infty\|\vecg\|_\infty2^{-h}|P_0|.\notag
\end{align}
Thus the remaining pairings tend to zero without any assertion about overlap between the enlargements of different parents. The stopping boxes over all stages form a $1/2$-sparse dyadic family, with witnesses given by each selected box after removing its selected children. This proves \eqref{eq:adapted-asymmetric-lifting}.

\par\Needspace{3\baselineskip}\smallskip\noindent\textbf{Step 5. Return to genuine adapted boxes.}
For each selected dyadic $Q$ use the genuine enlargement $Q^+$. The map $Q\mapsto Q^+$ is injective because the first-coordinate scale determines the generation and the center determines the box. If $E_Q\subseteq Q$ are the disjoint dyadic sparse witnesses, then
\begin{align}\label{eq:adapted-lift-genuine-witness}
E_Q\subseteq Q^+,\qquad
|E_Q|\ge\tfrac12|Q|\ge c_\gamma|Q^+|.
\end{align}
The convex body averages obey
\begin{align}\label{eq:adapted-lift-genuine-average}
|Q|\llangle\vecf\rrangle_{r,Q^+}
\circ\llangle\vecg\rrangle_{s,Q}
\le C_{\gamma,s}|Q^+|
\llangle\vecf\rrangle_{r,Q^+}
\circ\llangle\vecg\rrangle_{s,Q^+}.
\end{align}
Therefore $\Sc\coloneqq\{Q^+:Q\in\Sc_\star\}$ is a single $\delta_\gamma$-sparse family of genuine adapted boxes, and \eqref{eq:adapted-one-genuine} follows.
\end{proof}

\subsection{Proof of \Propref{prop:endpoint-flat}: the endpoint-contact criterion}\label{subsec:endpoint-flat-proof}
The thin-rectangle test below gives the contact-order threshold. We then check the polynomial and smooth examples.
\begin{proof}[Proof of \Propref{prop:endpoint-flat}]
Write $e\coloneqq\gamma'(1)/|\gamma'(1)|$ and choose the perpendicular unit vector $\nu$ from the statement. For $0<\delta<\delta_0$, let $R_\delta$ be the rectangle centered at the origin with side-lengths $\delta$ in direction $e$ and $\delta^\beta$ in direction $\nu$. Define
\begin{align}\label{eq:endpoint-flat-tests}
f_\delta&\coloneqq|R_\delta|^{-1/r}\mathbf1_{R_\delta},\qquad
 g_\delta\coloneqq|R_\delta|^{-1/s}\mathbf1_{R_\delta+\gamma(1)},\\
|R_\delta|&=\delta^{\beta+1},\qquad
\|f_\delta\|_{L^r}=\|g_\delta\|_{L^s}=1.\notag
\end{align}
The $C^1$ regularity and \eqref{eq:endpoint-contact-condition-r24} yield, for a sufficiently small constant $c_\gamma>0$ and $1-c_\gamma\delta\le t\le1$,
\begin{align}\label{eq:endpoint-polynomial-height}
\bigl|\langle\gamma(t)-\gamma(1),e\rangle\bigr|&\le\delta/2,\\
\bigl|\langle\gamma(t)-\gamma(1),\nu\rangle\bigr|&\le\delta^\beta/4.\notag
\end{align}
The intersection of the two rectangles is therefore controlled in each of their two orthogonal directions:
\begin{align*}
(\delta-\delta/2)(\delta^\beta-\delta^\beta/4)
=\frac38\delta^{\beta+1}=\frac38|R_\delta|.
\end{align*}
Consequently,
\begin{align}\label{eq:endpoint-contact-overlap-r24}
\bigl|(R_\delta+\gamma(t))\cap(R_\delta+\gamma(1))\bigr|
\ge\frac38|R_\delta|.
\end{align}
The terminal contribution is positive, while the negative half of the integral vanishes for sufficiently small $\delta$:
\begin{align*}
\operatorname{dist}\bigl(\gamma(1),\gamma([-1,0])\bigr)>0
&\Longrightarrow
\langle f_\delta(\cdot-\gamma(t)),g_\delta\rangle=0\quad(-1\le t\le0),\\
\langle f_\delta(\cdot-\gamma(t)),g_\delta\rangle&\ge0\quad(0<t\le1).
\end{align*}
The pairing is therefore absolutely convergent near $t=0$. From $t^{-1}\ge1$ on $[1-c_\gamma\delta,1]$ and \eqref{eq:endpoint-contact-overlap-r24} we obtain
\begin{align}\label{eq:endpoint-flat-lower}
\bigl|\langle H_{\gamma,\mathrm{loc}}f_\delta,g_\delta\rangle\bigr|
&\ge \frac{3c_\gamma}{8}\delta|R_\delta|^{-\theta}
=\frac{3c_\gamma}{8}\delta^{1-(\beta+1)\theta}.
\end{align}
The obstruction is visible directly in the exponent:
\begin{align*}
\theta>\frac1{\beta+1}
&\Longrightarrow 1-(\beta+1)\theta<0,\\
\lim_{\delta\downarrow0}|\langle H_{\gamma,\mathrm{loc}}f_\delta,g_\delta\rangle|&=\infty.
\end{align*}
For every $0<a<1$ choose $\delta$ with $1-c_\gamma\delta>a$. The same test gives
\begin{align*}
\|f_\delta\|_{L^r}=\|g_\delta\|_{L^s}=1,\qquad
\langle T_{\gamma,a}f_\delta,g_\delta\rangle
\gtrsim_\gamma\delta^{1-(\beta+1)\theta}\longrightarrow\infty,
\end{align*}
proving \eqref{eq:endpoint-terminal-unbounded-r37}.

The sparse form remains bounded. Since the supports are separated, every contributing adapted box satisfies
\begin{align*}
Q\cap\operatorname{supp}f_\delta\ne\varnothing,\quad
Q\cap\operatorname{supp}g_\delta\ne\varnothing
&\Longrightarrow \ell(Q)\ge c_\gamma'>0,\\
2^j\le\ell(Q)<2^{j+1}
&\Longrightarrow |Q|\eqsim_\gamma2^{3j},\\
\#\{Q\in\Sc:2^j\le\ell(Q)<2^{j+1},\ Q\text{ contributes}\}
&\lesssim_{\gamma,\eta}1.
\end{align*}
The last bound follows because the $\eta$-sparse witnesses are disjoint and lie in an adapted box of measure $O_\gamma(2^{3j})$. Consequently,
\begin{align}\label{eq:endpoint-flat-upper}
|Q|\langle f_\delta\rangle_{r,Q}\langle g_\delta\rangle_{s,Q}
&\le |Q|^{-\theta}\|f_\delta\|_{L^r}\|g_\delta\|_{L^s}
\lesssim_\gamma2^{-3j\theta}.
\end{align}
As $\theta>0$, we conclude that
\begin{align}\label{eq:endpoint-flat-sparse-upper}
\sup_{0<\delta<\delta_0}\ \sup_{\Sc\text{ $\eta$-sparse}}
\Lambda_{\Sc,r,s}(f_\delta,g_\delta)
\le C_{\gamma,r,s,\eta}<\infty.
\end{align}
For any fixed $J\in\mathbb N$ the right side of \eqref{eq:endpoint-no-finite-sparse-r37} is at most $JC_{\gamma,r,s,\eta}$ on this test pair, whereas \eqref{eq:endpoint-flat-lower} diverges. This proves \eqref{eq:endpoint-no-finite-sparse-r37}.

For the polynomial curve in \eqref{eq:endpoint-polynomial-curve}, direct differentiation gives
\begin{align}\label{eq:endpoint-polynomial-curvature}
\gamma_m'(t)&=\left(1,-\frac{(1-t)^{m-1}}{m-1}\right),\qquad
\gamma_m''(t)=(0,(1-t)^{m-2}).
\end{align}
Hence \eqref{eq:endpoint-flat-torsion} holds and $\gamma_m'(1)=(1,0)$. The contact condition is satisfied with $\beta=m$, since
\begin{align}\label{eq:endpoint-polynomial-contact-r24}
\bigl|\langle\gamma_m(1-u)-\gamma_m(1),(0,1)\rangle\bigr|
=\frac{u^m}{m(m-1)}.
\end{align}
The first coordinate also gives $\operatorname{dist}(\gamma_m(1),\gamma_m([-1,0]))\ge1$. Thus the geometric assertion applies. When $m=3$ and $r=s=20/13$, we have $\theta=3/10>1/4$ and
\begin{align}\label{eq:endpoint-interior-barycentric-r24}
\left(\frac{13}{20},\frac7{20}\right)
=\frac1{20}(0,0)+\frac9{10}\left(\frac23,\frac13\right)+\frac1{20}(1,1).
\end{align}
All three coefficients are positive, so the point belongs to $\inte(\Oc_2)$. Every interior point of $\Oc_2$ satisfies $\theta=1/r-1/s'>0$, so for each such point we can choose $m\ge3$ sufficiently large that $(m+1)\theta>1$.

Finally, define
\begin{align}\label{eq:endpoint-flat-curve}
w(t)&\coloneqq\begin{cases}
\exp\bigl(-1/(1-t)^2\bigr),&-1\le t<1,\\
0,&t=1,
\end{cases}\\
\phi(t)&\coloneqq\int_t^1(u-t)w(u)\,du,
\qquad\gamma_\infty(t)\coloneqq(t,\phi(t)-\phi(0)).\notag
\end{align}
Since $\phi''=w$, the torsion is positive for $-1<t<1$ and infinitely flat at $1$. For each $\beta>1$, the transverse displacement satisfies
\begin{align}\label{eq:endpoint-flat-height}
0\le\phi(1-u)&\le\frac{u^2}{2}\exp(-1/u^2)
\lesssim_\beta u^\beta,\qquad 0<u<1.
\end{align}
The other hypotheses follow from the first coordinate and $\phi''(0)>0$. For any $\theta>0$, choose $\beta>\max\{1,\theta^{-1}-1\}$ and apply the geometric assertion. This completes the proof.
\end{proof}

\subsection{Proof of \Propref{prop:B}: the adapted matrix characteristic}\label{subsec:metric-geometry}
To prove \Propref{prop:B}, we verify the doubling metric, full ball-basis characteristic and grouped dyadic geometry required by \cite[Theorem~1.3]{KakaroumpasNguyenVardakis2026}.
\begin{proof}[Proof of \Propref{prop:B}]
\noindent\textit{The adapted metric.} We first verify the metric geometry required by \Extref{KakaroumpasNguyenVardakis2026}{Theorem~1.3} \cite{KakaroumpasNguyenVardakis2026}. In fixed adapted coordinates the degrees are
\begin{align*}
(\alpha_1,\ldots,\alpha_d)=
\begin{cases}
\text{the monomial powers},&\text{for }H_\gamma,\\
(1,2,\ldots,d),&\text{for }H_{\gamma,\mathrm{loc}}.
\end{cases}
\end{align*}
The coordinate change has a fixed Jacobian and hence preserves all normalized averages up to an immaterial constant. Define
\begin{align}\label{eq:gamma-quasinorm}
\rho_\gamma(x-y)=\max_{1\le j\le d}
|x_j-y_j|^{1/\alpha_j},\qquad
B_\gamma(x,t)=\{y:\rho_\gamma(x-y)<t\}.
\end{align}
Set $\kappa\coloneqq \min\{1,\alpha_1,\ldots,\alpha_d\}$ and define
\begin{align*}
d_\gamma(x,y)\coloneqq \rho_\gamma(x-y)^\kappa
=\max_{1\le j\le d}|x_j-y_j|^{\kappa/\alpha_j}.
\end{align*}
Since $0<\kappa/\alpha_j\le1$, the coordinate snowflake metrics give
\begin{align*}
B_{d_\gamma}(x,t^\kappa)=B_\gamma(x,t),\qquad
\operatorname{side}_j(B_\gamma(x,t))=2t^{\alpha_j}.
\end{align*}
Their measures obey
\begin{align}\label{eq:gamma-doubling}
|B_\gamma(x,t)|=2^d t^{\alpha_1+\cdots+\alpha_d},\qquad
|B_\gamma(x,2t)|=2^{\alpha_1+\cdots+\alpha_d}
|B_\gamma(x,t)|.
\end{align}
Writing $B_{d_\gamma}(x,R)$ for a ball of radius $R$ in the metric $d_\gamma$, we have $B_{d_\gamma}(x,R)=B_\gamma(x,R^{1/\kappa})$ and hence
\begin{align}\label{eq:metric-doubling-r13}
|B_{d_\gamma}(x,2R)|=2^{(\alpha_1+\cdots+\alpha_d)/\kappa}|B_{d_\gamma}(x,R)|.
\end{align}
Thus $(\R^d,d_\gamma,dx)$ is a doubling metric measure space with the doubling constant expressed in its actual metric radius. The translated
boxes in $\Qc_\gamma$ have this ball geometry up to a fixed
change of radius and center. More precisely, if $Q$ has side-lengths
$\ell^{\alpha_j}$ and center $c_Q$, choose a fixed $c_0\in(0,1)$
so small that $2c_0^{\alpha_j}\le1$ for every $j$. Then
\begin{align}\label{eq:ball-sandwich-r4}
        B_\gamma(c_Q,c_0\ell)\subseteq Q
        \subseteq B_\gamma(c_Q,\ell),
        \qquad
        \frac{|B_\gamma(c_Q,\ell)|}{|Q|}=2^d.
\end{align}
\par\Needspace{3\baselineskip}\smallskip\noindent\textit{Complete ball-basis comparison.} Let $Q\subseteq R$ and suppose $|R|\le C|Q|$. For $x\in Q$,
\begin{align*}
\frac1{|Q|}\int_Q\|U(x)^{1/\beta}U(y)^{-1/\beta}\|^{\alpha'}\,dy
\le \frac{|R|}{|Q|}\frac1{|R|}\int_R\|U(x)^{1/\beta}U(y)^{-1/\beta}\|^{\alpha'}\,dy.
\end{align*}
Raising both sides to $\beta/\alpha'$, integrating over $Q$, and extending the outer integral to $R$ gives
\begin{align}\label{eq:matrix-characteristic-inclusion-r4}
    [U]_{A_{\alpha,\beta}}(Q)
    \le \left(\frac{|R|}{|Q|}\right)^{1+\beta/\alpha'}
    [U]_{A_{\alpha,\beta}}(R),
\end{align}
where the notation on the left means the individual expression in
\eqref{eq:matrix-Aab} before taking the supremum. We make both directions of the comparison explicit. By \eqref{eq:ball-sandwich-r4}, every genuine adapted box $Q$ lies inside $B_\gamma(c_Q,\ell(Q))$ and the ratio of their measures is $2^d$. Conversely, for every $B_\gamma(x,t)$ the genuine adapted box centered at $x$ with scale $2^{1/\alpha_1}$ times $t$ contains the ball and has measure bounded by a fixed multiple of its measure. Applying \eqref{eq:matrix-characteristic-inclusion-r4} in these two directions yields
\begin{align}\label{eq:matrix-ball-box-equivalence-r31}
\sup_{B\text{ adapted ball}}[U]_{A_{\alpha,\beta}}(B)
\eqsim_{\gamma,\alpha,\beta}
\sup_{Q\in\Qc_\gamma}[U]_{A_{\alpha,\beta}}(Q).
\end{align}
Since $B_{d_\gamma}(x,R)=B_\gamma(x,R^{1/\kappa})$, the suprema in \eqref{eq:matrix-ball-box-equivalence-r31} range over the complete ball basis required in \cite[Theorem~1.3]{KakaroumpasNguyenVardakis2026}.
For a dyadic box $P\in\mathscr D_\star(k)$, \Lemref{lem:corrected-grid} gives
\begin{align*}
  2^{-1}2^{k\alpha_j}<\operatorname{side}_j(P)\le 2^{k\alpha_j},
  \qquad 1\le j\le d.
\end{align*}
Thus there are fixed $c_1,c_2>0$, depending only on the anisotropic degrees, such that
\begin{align*}
 B_\gamma(c_P,c_1 2^k)\subseteq P\subseteq B_\gamma(c_P,c_2 2^k).
\end{align*}
For the local torsion curve this calculation takes place in its fixed origin Taylor frame with degrees $1,\ldots,d$. The temporary frames used to establish the scalar single-scale bound in \cite[Lemma~3.2]{CladekOu2018} do not change the stopping rectangles. In the metric $d_\gamma$ the generation-$k$ boxes have diameter at most $C_\gamma 2^{k\kappa}$ and contain balls of radius $c_\gamma 2^{k\kappa}$. They are nested by \eqref{eq:alternating-nesting-proof}, and \eqref{eq:adapted-parent-ratio} controls the parent-child measure ratio.

\par\Needspace{3\baselineskip}\smallskip\noindent\textit{Grouped dyadic geometry.} We verify the dyadic-system hypotheses in \cite[Subsection~2.2]{KakaroumpasNguyenVardakis2026}. Set $R_k\coloneqq2^{\kappa k}$ and define the geometric constants
\begin{align}\label{eq:matrix-grid-constants-r31}
c_*\coloneqq\min_{1\le i\le d}4^{-\kappa/\alpha_i},
\qquad D_*\coloneqq\max_{1\le i\le d}2^{-\kappa/\alpha_i}<1.
\end{align}
For every $Q\in\mathscr D_\star(k)$, the bounds $2^{-1}2^{k\alpha_i}<\operatorname{side}_i(Q)\le2^{k\alpha_i}$ imply
\begin{align}\label{eq:matrix-grid-ball-bounds-r31}
B_{d_\gamma}(c_Q,c_*R_k)\subseteq Q
\subseteq B_{d_\gamma}(c_Q,R_k).
\end{align}
The same coordinate calculation gives $d_\gamma(c_Q,c_P)\le D_*R_{k(P)}$ whenever $Q\subseteq P$. Choose $m_0\in\mathbb N$ so large that $\delta\coloneqq2^{-m_0\kappa}\le c_*/24$, and for $b\in\{0,\ldots,m_0-1\}$ set
\begin{align}\label{eq:matrix-grouping-r31}
\mathscr D_b^j\coloneqq\mathscr D_\star(b-m_0j),\qquad
c_{0,b}\coloneqq c_*2^{\kappa b},\qquad
C_{0,b}\coloneqq2\,2^{\kappa b}.
\end{align}
A box $Q\in\mathscr D_b^j$ has $R_{k(Q)}=2^{\kappa b}\delta^j$, and the estimates above yield
\begin{align}\label{eq:matrix-christ-ball-hypotheses-r31}
B_{d_\gamma}(c_Q,\tfrac{c_{0,b}}3\delta^j)
&\subseteq Q\subseteq
B_{d_\gamma}(c_Q,2C_{0,b}\delta^j),\\
12C_{0,b}\delta&=24\,2^{\kappa b}\delta
\le c_{0,b}.\notag
\end{align}
The balls in \eqref{eq:matrix-christ-ball-hypotheses-r31} satisfy the required nesting. Indeed, if $Q\in\mathscr D_b^{j+1}$ and $P\in\mathscr D_b^j$ with $Q\subseteq P$, then
\begin{align}\label{eq:grouped-outer-ball-inclusion-r17}
 d_\gamma(z,c_P)&\le d_\gamma(z,c_Q)+d_\gamma(c_Q,c_P)\notag\\
 &<4\delta R_{k(P)}+D_*R_{k(P)}
 <4R_{k(P)}
\end{align}
for $z\in B_{d_\gamma}(c_Q,4R_{k(Q)})$, since $\delta\le c_*/24<1/2$. For more distant descendants, iterate this containment. The centers of distinct generation-$k$ cells are separated by at least $\min_i2^{-\kappa/\alpha_i}R_k\ge c_*R_k$, and their cells cover the entire space up to null boundaries. These facts also give the center separation and covering properties \cite[(2.2)--(2.3)]{KakaroumpasNguyenVardakis2026}. Thus each $\mathscr D_b$ satisfies the nestedness, covering, inner-ball, outer-ball and outer-ball nesting requirements of \cite[Theorem~2.1]{KakaroumpasNguyenVardakis2026}. The interiors may be used for the open dyadic sets, since their boundaries have Lebesgue measure zero. The constants in \eqref{eq:matrix-grouping-r31} vary with $b$, but the bounds and the number $m_0$ depend only on $\gamma$. Every $\Sc\subseteq\mathscr D_\star$ decomposes into $m_0$ sparse subfamilies $\Sc\cap\mathscr D_b$, and the original sparse witnesses remain valid.

By \eqref{eq:matrix-ball-box-equivalence-r31}, the complete ball-basis matrix characteristic required by the cited theorem is bounded by $C_{\gamma,p,r,s}[W]_{p,r,s}$.

\par\Needspace{3\baselineskip}\smallskip\noindent\textit{The matrix embedding.} With the doubling and dyadic properties verified, apply \Extref{KakaroumpasNguyenVardakis2026}{Theorem~1.3} \cite{KakaroumpasNguyenVardakis2026} to the doubling adapted space. In that theorem choose $p_0=r$, $q_0=s'$, $t=p/r$ and $\sigma=s'/p$. Its matrix condition agrees with \eqref{eq:matrix-characteristic}, since
\begin{align*}
 a=\frac{pr}{p-r},\qquad b=\frac{ps'}{s'-p},\qquad
 \sigma'=\frac{s'}{s'-p},\qquad
 [W^{\sigma'}]_{A_{a',b}}=[W]_{p,r,s}.
\end{align*}
The theorem provides the power
\begin{align}\label{eq:matrix-exponents-r31}
\frac1{\sigma'(p-r)}+\frac1{p'}
=\frac{s'-p}{s'(p-r)}+\frac1{p'}.
\end{align}
The embedding requires $1<a'\le b<\infty$, which follows from $1\le r<p<s'$:
\begin{align}\label{eq:matrix-exponent-order-r38}
a=\frac{pr}{p-r}\ge p',\qquad
1<a'\le p<b=\frac{ps'}{s'-p}<\infty.
\end{align}
Its joint integrability criterion \cite[Lemma~5.2]{KakaroumpasNguyenVardakis2026}, applied with $V=W^{\sigma'}$, gives
\begin{align*}
V^{1/b}=W^{1/p},\qquad V^{-1/b}=W^{-1/p},\qquad
W^{-1/p}\in L^a_{\mathrm{loc}},\qquad W^{1/p}\in L^b_{\mathrm{loc}}.
\end{align*}
The exponents satisfy
\begin{align}\label{eq:matrix-local-exponents-r31}
a=\frac{pr}{p-r}>r,
\qquad b=\frac{ps'}{s'-p}>s,
\end{align}
hence
\begin{align*}
W^{-1/p}\vecF&\in L^r_{\mathrm{loc}},\qquad
W^{1/p}\vecG\in L^s_{\mathrm{loc}},\\
a\ge p'&\Longrightarrow W^{-p'/p}\in L^1_{\mathrm{loc}}.
\end{align*}
The averages in \eqref{eq:matrix-embedding} and duality in \eqref{eq:duality} are well defined. The dilation in \cite[(2.4)]{KakaroumpasNguyenVardakis2026} is the metric neighborhood
\begin{align}\label{eq:matrix-source-neighborhood-r31}
\lambda Q\coloneqq\{z:d_\gamma(z,Q)\le(\lambda-1)\ell_b(Q)\},
\qquad\ell_b(Q)=2C_{0,b}\delta^j=4R_{k(Q)}
\end{align}
for $Q\in\mathscr D_b^j$. If $z\in Q^+$, then $c_Q\in Q$ and the explicit side-lengths of $Q^+$ imply
\begin{align}\label{eq:matrix-plus-source-containment-r31}
d_\gamma(z,Q)\le d_\gamma(z,c_Q)
&\le D_* A_\gamma^\kappa R_{k(Q)}
 < A_\gamma^\kappa R_{k(Q)}.
\end{align}
Choose once and for all $\lambda\ge\max\{2,1+2/\delta,1+A_\gamma^\kappa\}$. Then $Q^+\subseteq\lambda Q$ for every $Q$ in every grouped system. Moreover, the triangle inequality and \eqref{eq:matrix-grid-ball-bounds-r31} give
\begin{align}\label{eq:matrix-dilation-volume-r31}
Q\subseteq\lambda Q\subseteq
B_{d_\gamma}(c_Q,(4\lambda-3)R_{k(Q)}),
\qquad |\lambda Q|\eqsim_{\gamma,\lambda}|Q|.
\end{align}
For any measurable sets $E\subseteq D$ of positive finite measure and every $1\le a<\infty$, extension of scalar test functions from $E$ to $D$ gives
\begin{align}\label{eq:dilated-average-inclusion}
\llangle\vech\rrangle_{a,E}\subseteq
\left(\frac{|D|}{|E|}\right)^{1/a}\llangle\vech\rrangle_{a,D}.
\end{align}
We apply this inclusion with $(E,D)=(Q^+,\lambda Q)$ and $(Q,\lambda Q)$. By \eqref{eq:matrix-dilation-volume-r31}, their two scalar factors are uniformly bounded. The published matrix theorem controls the convex body form with \emph{both} averages on $\lambda Q$, so it dominates our asymmetric form. Precisely, \Extref{KakaroumpasNguyenVardakis2026}{Theorem~1.3} \cite{KakaroumpasNguyenVardakis2026}, \eqref{eq:matrix-ball-box-equivalence-r31} and the decomposition into the $m_0$ grouped sparse families yield
\begin{align}\label{eq:matrix-embedding-applied-r31}
\sum_{Q\in\Sc}|Q|\llangle W^{-1/p}\vecF\rrangle_{r,Q^+}\circ\llangle W^{1/p}\vecG\rrangle_{s,Q}
\lesssim_{\gamma,n,p,r,s,\eta}
[W]_{p,r,s}^{\frac{s'-p}{s'(p-r)}+\frac1{p'}}
\|\vecF\|_{L^p}\|\vecG\|_{L^{p'}}.
\end{align}
This proves \eqref{eq:matrix-embedding}.
\end{proof}

\medskip
\noindent{\bf Funding }
Not applicable.
\medskip

\noindent{\bf Data Availability} This manuscript has no associated data.

\medskip
\noindent{\bf\Large Declarations}
\medskip

\noindent{\bf Conflicts of Interest} The author states that there are no conflicts of interest.

\medskip
\noindent{\bf AI disclosure.}
During the preparation of this manuscript, the author used ChatGPT to discuss possible proof strategies, assist in drafting and checking parts of the mathematical arguments, organize references, and improve the clarity and readability of the exposition. All mathematical statements and proofs in the final manuscript were written, reviewed, and verified by the author, who takes full responsibility for the content of this manuscript.

\end{document}